\documentclass[a4paper,12pt]{article}
\usepackage{amsmath,amsthm}
\usepackage{amssymb}
\usepackage{latexsym}
\usepackage{graphicx}
\usepackage{tikz}
\usepackage{scalefnt}
\usepackage{lmodern}
\usepackage{float}
\usetikzlibrary{positioning}
\usetikzlibrary{shapes,arrows}
\usepackage{subfigure}
\usepackage{pgfplots}
\usepackage{MnSymbol}
\usepackage{wrapfig}
\usepackage[pdfpagemode=UseNone,pdfstartview=FitH,hypertexnames=false]{hyperref}
\newtheorem{theorem}{Theorem}[section]
\newtheorem{corollary}{Corollary}[section]
\newtheorem{remark}{Remark}[section]
\newtheorem{example}{Example}[section]
\newtheorem{lemma}{Lemma}[section]
\newtheorem{conjecture}{Conjecture}[section]

\newtheorem{proposition}{Proposition}[section]
  
\newcommand*{\Zset}{\mathbb{Z}}  
  
\newcommand*{\Rset}{\mathbb{R}}  
\newcommand*{\Cset}{\mathbb{C}}  

\begin{document}
\title{{\Large \bf Studying the singularity of LCM-type matrices via semilattice structures and their Möbius functions}}
\date{March 20, 2014}
\author{{\sc Mika Mattila, Pentti Haukkanen}
\\{\sc and Jori Mäntysalo}
\\School of Information Sciences
\\FI-33014 University of Tampere, Finland
\\E-mail:  mika.mattila@uta.fi, pentti.haukkanen@uta.fi,\\ jori.mantysalo@uta.fi}
\maketitle
\setcounter{section}{0}
\setcounter{equation}{0}
\begin{abstract}
The invertibility of LCM matrices and their Hadamard powers have been studied a lot over the years by many authors. Bourque and Ligh conjectured in 1992 that the LCM matrix $[S]=[[x_i,x_j]]$ on any GCD closed set $S=\{x_1,x_2,\ldots,x_n\}$ is invertible, but in 1997 this was proven false. However, currently there are many open conjectures concerning LCM matrices and their Hadamard powers presented by Hong. In this paper we utilize lattice-theoretic structures and the Möbius function to explain the singularity of classical LCM matrices and their Hadamard powers. At the same time we end up disproving some of Hong's conjectures. Elementary mathematical analysis is applied to prove that for most semilattice structures there exist a set $S=\{x_1,x_2,\ldots,x_n\}$ of positive integers and a real number $\alpha>0$ such that $S$ possesses this structure and the power LCM matrix $[[x_i,x_j]^\alpha]$ is singular.
\end{abstract}
\medskip\noindent{\it Key words and phrases}:\\LCM matrix, GCD matrix, Smith determinant, Möbius function,\\ Meet semilattice, Singular number\\
\noindent{\it AMS Subject Classification:}\\11C20, 15B36

\section{Introduction}
\setcounter{equation}{0}

The study of GCD and LCM matrices was initiated by famous number theorist H. J. S. Smith \cite{S} in 1876. Smith calculated the determinant of the basic GCD matrix with the greatest common divisor of $i$ and $j$ as its $ij$-entry. In addition, Smith derived determinant formulas for more general GCD and LCM matrices with $(x_i,x_j)$ or $[x_i,x_j]$ as its $ij$-entry and showed that the GCD matrix $(S)$ and the LCM matrix $[S]$ are nonsingular on factor closed sets $S$. He also studied GCD and LCM matrices associated with arithmetical function $f$, where the $ij$ entries are $f((x_i,x_j))$ and $f([x_i,x_j])$, respectively. Determinants of GCD-related matrices were studied in dozens of papers during the 20th century (see e.g. the references in \cite{HWS}), but Bourque and Ligh \cite{BL} were the first ones to bring special attention to the invertibility properties of LCM matrices as they conjectured that the LCM matrix of a GCD closed set is always invertible. Shen \cite{Sh} went even further and conjectured that if the set $S$ is GCD closed and $r\neq0$, then the power LCM matrix $[[x_i,x_j]^r]$ is nonsingular.

Haukkanen et al. \cite{HWS} soon showed that the Bourque-Ligh conjecture (and also Shen's conjecture in the case $r=1$) is false by finding a counterexample with $9$ elements. Two years later Hong \cite{Ho} found a counterexample with $8$ elements and proved number-theoretically that the Bourque-Ligh conjecture holds for $n\leq7$ and does not hold in general for $n\geq8$ (there is also a recent paper by Korkee et al. \cite{KMH} which gives another, a lattice-theoretic proof for this fact). Subsequently Hong published many papers regarding power GCD and power LCM matrices (see e.g. \cite{Ho2, Ho3, Ho4}). Hong also ended up presenting several conjectures on his own about the nonsingularity of power GCD and power LCM matrices. For example, in \cite{Ho2} Hong conjectured that if $S$ is a GCD closed set of odd integers, then every power LCM matrix of the set $S$ with nonzero exponent is nonsingular.

During the last ten years there has not been much progress on proving or disproving Hong's conjectures, and they all have remained open. One of the few advances was Li's article \cite{Li}, which provided some support to two of the conjectures. In this article we improve this situation by showing that some of Hong's conjectures are in fact false. This is done by using lattice-theoretic methods.

In Section 2 we introduce some key definitions and preliminary results needed in the following sections. In Section 3 we study the zeros of the Möbius function in a given meet semilattice, which gives us the leverage to analyze the product expression of the determinant of LCM-type matrices. In Section 4 we apply the mathematics software Sage \cite{Stein} to show that every $8$-element GCD closed set $S$, for which the LCM matrix $[S]$ is singular, has the same semilattice structure. We also construct a GCD closed set $S$ of odd numbers such that the LCM matrix $[S]$ is singular. In Section 5 we prove that for most semilattice structures $(L,\preceq)$ there exists a set $S=\{x_1,x_2,\ldots,x_n\}$ of positive integers and a positive real number $\alpha$ such that $(S,|)\cong(L,\preceq)$ and the power LCM matrix $[S]_{N^\alpha}:=[[x_i,x_j]^\alpha]$ is singular. We also point out a connection between the $\wedge$-tree structure of $(L,\preceq)$, the nonpositiveness of the nontrivial values of the Möbius function $\mu_L$ and the nonsingularity of the power LCM matrices $[S]_{N^\alpha}$ for all $(S,|)\cong(L,\preceq)$ and $\alpha>0$. In Section 6 we shall encounter several conjectures by Hong and give conclusive answers to some of them.

\section{Preliminaries}

If $(P,\preceq)$ is a meet semilattice, $f$ is a function $P\to\Cset$ and $S=\{x_1,\ldots,x_n\}$ is a subset of $P$ with distinct elements arranged so that $x_i\preceq x_j\Rightarrow i\leq j$, then the \emph{meet matrix of the set $S$ with respect to the function $f$} has $f(x_i\wedge x_j)$ as its $ij$-entry. This matrix is usually denoted by $(S)_f$. Similarly, if $(P,\preceq)$ is a join semilattice and $f$ and $S$ are as above, then the \emph{join matrix of the set $S$ with respect to the function $f$} has $f(x_i\vee x_j)$ as its $ij$-entry. For this join matrix we use the notation $[S]_f$.

In the special case when $(P,\preceq)=(\Zset_+,|)$ and $f$ is an arithmetical function the meet and join matrices become the so-called \emph{GCD and LCM matrices with respect to the arithmetical function} $f$, respectively. Moreover, if we set $f=N^\alpha$, where $N^\alpha(m)=m^\alpha$ for all $m\in \Zset_+$, the matrices $(S)_f$ and $[S]_f$ become the power-GCD and power-LCM matrices with $(x_i,x_j)^\alpha$ and $[x_i,x_j]^\alpha$ as their $ij$-entries, respectively. And in the case when $\alpha=1$ we denote $N^1=N$ and obtain the usual GCD and LCM matrices with $(x_i,x_j)$ and $[x_i,x_j]$ as their $ij$-entries, respectively. The usual GCD matrix of the set $S$ is denoted by $(S)$, and the usual LCM matrix by $[S]$.

\begin{remark}
It is often convenient to assume that $x_i\preceq x_j\Rightarrow i\leq j$ (in the case of meet and join matrices) or that $x_1\leq x_2\leq\cdots\leq x_n$ (in the case of GCD and LCM matrices). However, the indexing of the elements of the set $S$ does not affect on the invertibility of the corresponding meet or join matrix, see e.g. \cite[Remark 2.1]{MH14}. Since in this paper we are only interested in the singularity or nonsingularity of these matrices, in most of the cases we could also do without this assumption. 
\end{remark}

We develop further the lattice-theoretic method adopted in \cite{KMH}. This time we will focus solely on power-LCM and power-GCD matrices. Throughout this paper, let $S=\{x_1,\ldots,x_n\}$ be a GCD closed set of positive integers. By denoting $S_i=\{x_1,\ldots,x_i\}$ we obtain a chain of GCD closed sets $S_1\subset S_2\subset\cdots\subset S_n=S$. It should be noted that every set $S_i$ is also trivially lower closed in $(S,|)$. This observation enables us to use the Möbius function $\mu_S$ of the set $S$, which can be given recursively as
\begin{align*}
&\mu_S(x_i,x_i)=1,\\
&\mu_S(x_i,x_j)=-\sum_{\substack{x_i\,|\,x_k\,|\,x_j\\x_k\neq x_i}}\mu_S(x_i,x_k)=-\sum_{\substack{x_i\,|\,x_k\,|\,x_j\\ x_k\neq x_j}}\mu_S(x_k,x_j).
\end{align*}

Since
\[
[S]_{N^\alpha}=\mathrm{diag}(x_1^\alpha,\ldots,x_n^\alpha)(S)_{\frac{1}{N^\alpha}}\mathrm{diag}(x_1^\alpha,\ldots,x_n^\alpha),
\]
it follows that $[S]_{N^\alpha}$ is singular if and only if $(S)_{\frac{1}{N^\alpha}}$ is singular. Further, since the set $S$ is GCD closed, we may define the function $\Psi_{S,\frac{1}{N^\alpha}}$ on $S$ as
\begin{equation}\label{eq:psi}
\Psi_{S,\frac{1}{N^\alpha}}(x_i)=\sum_{x_k\,|\,x_i}\frac{\mu_S(x_k,x_i)}{x_k^\alpha}
\end{equation}
(if the set $S$ was not GCD closed we would have to define this function on an auxiliary set $D$ such that $(x_i,x_j)\in D$ for all $x_i,x_j\in S$, as done in \cite{ATH}). By a well-known determinant formula (see e.g. \cite[Theorem 4.2]{ATH}) we now have
\begin{equation}\label{eq:det}
\det (S)_{\frac{1}{N^\alpha}}=\Psi_{S,\frac{1}{N^\alpha}}(x_1)\Psi_{S,\frac{1}{N^\alpha}}(x_2)\cdots\Psi_{S,\frac{1}{N^\alpha}}(x_n).
\end{equation}
Thus we may conclude the following result.

\begin{proposition}\label{psi-lause}
The matrices $[S]_{N^\alpha}$ and $(S)_{\frac{1}{N^\alpha}}$ are both invertible if and only if $\Psi_{S,\frac{1}{N^\alpha}}(x_i)\neq0$ for all $i=1,\ldots,n$.
\end{proposition}

\begin{remark}
Proposition \ref{psi-lause} shows that the Möbius function plays a crucial role in invertibility of power LCM and GCD matrices of GCD closed sets. For material on the Möbius function we refer to \cite{A, R, St}.
\end{remark}

\begin{remark}
By applying the other equation in \cite[Theorem 4.2]{ATH} we obtain
\[
\Psi_{S,\frac{1}{N^\alpha}}(x_i)=\sum_{\substack{z\,|\,x_i\\z\,\nmid\,x_j\ \mathrm{for}\ j<i}}\sum_{w\,|\,z}\frac{1}{w^\alpha}\mu\left(\frac{z}{w}\right)=\sum_{\substack{z\,|\,x_i\\z\,\nmid\,x_j\ \mathrm{for}\ j<i}}\left(\frac{1}{N^\alpha}*\mu\right)(z),
\]
where $\mu$ is the number-theoretic Möbius function and $*$ is the Dirichlet convolution. Therefore $\Psi_{S,\frac{1}{N^\alpha}}(x_i)$ is equal to $\alpha_i$ (or $\alpha_i(x_1,\ldots,x_k)$), which appears in many texts by Bourque and Ligh and Hong (see e.g. \cite{BL} and \cite{Ho3}), but here we only use a different method for calculating it.   
\end{remark}

Finally we need the following proposition.

\begin{proposition}\label{aputulos}
Let $T=\{t_1,\ldots,t_m\}$ be any subset of $S$ with $t_i\,|\,t_j\Rightarrow i\leq j$. If the poset $(T,|)$ belongs to one of the classes presented in Figure \ref{fig:aputulos}, then
\[
\sum_{i=1}^m\frac{a_i}{t_i}>0,
\]
where the $a_i$'s are the coefficients found in Figure \ref{fig:aputulos} next to each element.
\end{proposition}

\begin{proof}
Consider first Figure \ref{fig:aputulos}(a). Now $(T,|)\in{\cal S}_{1,2}$ and $t_1\,|\,t_2$, and thus clearly $\frac{1}{t_1}-\frac{1}{t_2}>0$. Consider next Figures \ref{fig:aputulos}(b)-\ref{fig:aputulos}(e). Then $m\geq4$, $t_1\,|\,t_2,\ldots,t_{m-1}$ and $t_2,\ldots,t_{m-1}\,|\,t_m$. In this case
\[
\frac{m-3}{t_1}+\frac{1}{t_m}-\sum_{k=2}^{m-1}\frac{1}{t_k}=\frac{1}{t_m}+\frac{1}{t_1}\left((m-3)-\sum_{k=2}^{m-1}\frac{t_1}{t_k}\right)>0,
\]
since
\[
\sum_{k=2}^{m-1}\frac{t_1}{t_k}<\sum_{k=2}^{m-1}\frac{1}{2}=\frac{m-2}{2}\leq m-3.
\]
\end{proof}

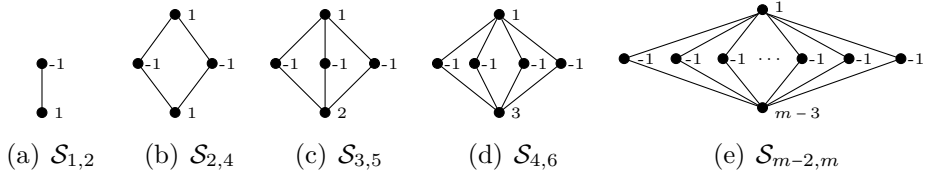
\begin{figure}[htb!]
\centering
\subfigure[${\cal S}_{1,2}$]
{{\scalefont{0.5}
\begin{tikzpicture}[scale=0.65]
\draw (0,1)--(0,1.9);
\draw [fill] (0,1) circle [radius=0.1];
\draw [fill] (0,2) circle [radius=0.1];
\node [right] at (0.1,1) {$1$};
\node [right] at (0,2) {\textrm{-}$1$};
\node [right] at (1,2) {\ };
\node [right] at (-1,2) {\ };
\end{tikzpicture}}
}
\subfigure[${\cal S}_{2,4}$]
{{\scalefont{0.5}
\begin{tikzpicture}[scale=0.65]
\draw (1,0)--(0.25,1)--(0.92,1.92);
\draw (1,0)--(1.75,1)--(1.08,1.92);
\draw [fill] (0.25,1) circle [radius=0.1];
\draw [fill] (1.75,1) circle [radius=0.1];
\draw [fill] (1,0) circle [radius=0.1];
\draw [fill] (1,2) circle [radius=0.1];
\node [right] at (1.1,0) {$1$};
\node [right] at (0.25,1) {\textrm{-}$1$};
\node [right] at (1.75,1) {\textrm{-}$1$};
\node [right] at (1.1,2) {$1$};
\end{tikzpicture}}
}
\subfigure[${\cal S}_{3,5}$]
{{\scalefont{0.5}
\begin{tikzpicture}[scale=0.65]
\draw (1,0)--(0,1)--(0.92,1.92);
\draw (1,0)--(1,1.9);
\draw (1,0)--(2,1)--(1.08,1.92);
\draw [fill] (0,1) circle [radius=0.1];
\draw [fill] (1,1) circle [radius=0.1];
\draw [fill] (2,1) circle [radius=0.1];
\draw [fill] (1,0) circle [radius=0.1];
\draw [fill] (1,2) circle [radius=0.1];
\node [right] at (1.1,0) {$2$};
\node [right] at (0,1) {\textrm{-}$1$};
\node [right] at (1,1) {\textrm{-}$1$};
\node [right] at (2,1) {\textrm{-}$1$};
\node [right] at (1.1,2) {$1$};
\end{tikzpicture}}
}
\subfigure[${\cal S}_{4,6}$]
{{\scalefont{0.5}
\begin{tikzpicture}[scale=0.65]
\draw (1,0)--(-0.25,1)--(0.92,1.92);
\draw (1,0)--(0.5,1)--(0.95,1.9);
\draw (1,0)--(1.5,1)--(1.05,1.9);
\draw (1,0)--(2.25,1)--(1.08,1.92);
\draw [fill] (-0.25,1) circle [radius=0.1];
\draw [fill] (0.5,1) circle [radius=0.1];
\draw [fill] (1.5,1) circle [radius=0.1];
\draw [fill] (1,0) circle [radius=0.1];
\draw [fill] (2.25,1) circle [radius=0.1];
\draw [fill] (1,2) circle [radius=0.1];
\node [right] at (1.1,0) {$3$};
\node [right] at (-0.25,1) {\textrm{-}$1$};
\node [right] at (1.5,1) {\textrm{-}$1$};
\node [right] at (0.5,1) {\textrm{-}$1$};
\node [right] at (2.25,1) {\textrm{-}$1$};
\node [right] at (1.1,2) {$1$};
\end{tikzpicture}}
}
\subfigure[${\cal S}_{m-2,m}$]
{{\scalefont{0.5}
\begin{tikzpicture}[scale=0.65]
\draw (1,0)--(-1.8,1)--(0.92,1.92);
\draw (1,0)--(-0.75,1)--(0.92,1.92);
\draw (1,0)--(0.2,1)--(0.92,1.92);
\draw (1,0)--(1.8,1)--(1.08,1.92);
\draw (1,0)--(2.75,1)--(1.08,1.92);
\draw (1,0)--(3.8,1)--(1.08,1.92);
\draw [fill] (-1.8,1) circle [radius=0.1];
\draw [fill] (3.8,1) circle [radius=0.1];
\draw [fill] (0.2,1) circle [radius=0.1];
\draw [fill] (1.8,1) circle [radius=0.1];
\draw [fill] (1,0) circle [radius=0.1];
\draw [fill] (2.75,1) circle [radius=0.1];
\draw [fill] (-0.75,1) circle [radius=0.1];
\draw [fill] (1,2) circle [radius=0.1];
\node [right] at (1.1,-0.1) {$m-3$};
\node [right] at (0.2,1) {\textrm{-}$1$};
\node at (1,1) {\ \ $\ldots$};
\node [right] at (-0.7,1) {\textrm{-}$1$};
\node [right] at (2.75,1) {\textrm{-}$1$};
\node [right] at (1.8,1) {\textrm{-}$1$};
\node [right] at (1.1,2.05) {$1$};
\node [right] at (-1.8,1) {\ \textrm{-}$1$};
\node [right] at (3.8,1) {\textrm{-}$1$};
\end{tikzpicture}}
}
\caption{The lattice classes of Proposition \ref{aputulos}. The coefficients $a_i$ are next to each element.}\label{fig:aputulos}
\end{figure}

\section{On the zeros of the Möbius function of a meet semilattice}

Before we can begin our study of singular LCM matrices we need to prove the following lemma. It is going to tell us something important about the zeros of the Möbius function of a finite meet semilattice. 

\begin{lemma}\label{Möbius0}
Let $(L,\preceq)$ be a finite meet semilattice, $x\in L$ and $C_L(x)=\{y\in L\ |\ x\text{ covers }y\}$. Denote $\xi_L(x)=\bigwedge C_L(x)$ if $C_L(x)\neq\emptyset$ and $\xi_L(x)=x$ if $C_L(x)=\emptyset$. If $$z\not\in\boldsymbol{\lsem}\xi_L(x),x\boldsymbol{\rsem}:=\{w\in L\ |\ \xi_L(x)\preceq w\preceq x\},$$ then $\mu_L(z,x)=0.$ 
\end{lemma}
\begin{proof}
If $C_L(x)=\emptyset$, then $x=\min L$ and we have $\xi_L(x)=x$. Trivially $\mu_L(z,x)=0$ for all $z\not\in\boldsymbol{\lsem}x,x\boldsymbol{\rsem}$, so we may assume that $C_L(x)\neq\emptyset$. Let $m$ denote the number of elements in $C_L(x)$ ($m\geq1$). Suppose that $z\not\in\boldsymbol{\lsem}\xi_L(x),x\boldsymbol{\rsem}$. Clearly the claim is true if $z\not\preceq x$, so we may assume that $z\preceq x$. Let $\eta(z)$ denote the number of elements $y_i\in C_L(x)$ such that $z\prec y_i$, and let $y_1,y_2,\ldots,y_{\eta(z)}$ be these elements (thus $C_{\boldsymbol{\lsem}z,x\boldsymbol{\rsem}}(x)=\{y_1,y_2,\ldots,y_{\eta(z)}\}$). In addition, $\xi_{\boldsymbol{\lsem}z,x\boldsymbol{\rsem}}(x)=y_1\wedge y_2\wedge\cdots\wedge y_{\eta(z)}$ (clearly $\xi_{\boldsymbol{\lsem}z,x\boldsymbol{\rsem}}(x)\in\boldsymbol{\lsem}\xi_L(x),x\boldsymbol{\rsem}$). We apply double induction: first induction on the size of $C_L(x)$ and then induction on the size of the interval $\boldsymbol{\lsem}z,\xi_{\boldsymbol{\lsem}z,x\boldsymbol{\rsem}}(x)\boldsymbol{\lsem}$.

Our base case is that the set $C_L(x)$ has one element, i.e. $m=1$. Suppose first that there is only one element on the interval $\boldsymbol{\lsem}z,\xi_{\boldsymbol{\lsem}z,x\boldsymbol{\rsem}}(x)\boldsymbol{\lsem}=\boldsymbol{\lsem}z,y_1\boldsymbol{\lsem}$. This element has to be $z$ itself. In this case the interval $\boldsymbol{\lsem}z,x\boldsymbol{\rsem}$ is equal to the chain $z\prec y_1\prec x$, and clearly 
\[
\mu_L(z,x)=-\sum_{z\prec v\preceq x}\mu_L(v,x)=-(\mu_L(y_1,x)+\mu_L(x,x))=-(-1+1)=0.
\]
Next we consider the case $m=1$ and there are more than one elements on the interval $\boldsymbol{\lsem}z,\xi_{\boldsymbol{\lsem}z,x\boldsymbol{\rsem}}(x)\boldsymbol{\lsem}=\boldsymbol{\lsem}z,y_1\boldsymbol{\lsem}=\boldsymbol{\lsem}z,\xi_L(x)\boldsymbol{\lsem}$. Here our secondary induction hypothesis is that if $u\not\in\boldsymbol{\lsem}\xi_L(x),x\boldsymbol{\lsem}$, $m=1$ and there are less than $k(\geq2)$ elements on the interval $\boldsymbol{\lsem}u,\xi_L(x)\boldsymbol{\lsem}$, then $\mu_L(u,x)=0$. Suppose that there are $k$ elements on the interval $\boldsymbol{\lsem}z,\xi_L(x)\boldsymbol{\lsem}$. Since in this case $\boldsymbol{\rsem}z,x\boldsymbol{\rsem}=\boldsymbol{\rsem}z,\xi_L(x)\boldsymbol{\lsem}\cup\boldsymbol{\lsem}\xi_L(x),x\boldsymbol{\rsem}$, we have
\[
\mu_L(z,x)=-\left(\sum_{z\prec v\prec \xi_{\boldsymbol{\lsem}z,x\boldsymbol{\rsem}}(x)}\overbrace{\mu_L(v,x)}^{\substack{=0\text{ by induction}\\\text{hypothesis}}}\right)-\left(\overbrace{ \sum_{\xi_{\boldsymbol{\lsem}z,x\boldsymbol{\rsem}}(x)\preceq v\preceq x}\mu_L(v,x)}^{\substack{ =\delta_L(\xi_{\boldsymbol{\lsem}z,x\boldsymbol{\rsem}}(x),x)=0,\\\text{ since }\xi_{\boldsymbol{\lsem}z,x\boldsymbol{\rsem}}(x)\neq x}}\right)=-0-0=0.
\]
Thus our base case is complete.

Now let $m>1$. Our primary induction hypothesis is that for all semilattices $L$ in which $x$ covers less than $m$ elements we have $\mu_L(u,x)=0$ for all $u\not\in\boldsymbol{\lsem}\xi_L(x),x\boldsymbol{\rsem}$.

Suppose first that $z\not\in\boldsymbol{\lsem}\xi_L(x),x\boldsymbol{\rsem}$ is fixed and $\eta(z)<m$. When calculating the value $\mu_L(z,x)$ we may restrict ourselves to the meet semilattice $\boldsymbol{\lsem}z,x\boldsymbol{\rsem}$. In this structure $z$ precedes and $x$ covers less than $m$ of the elements of $C_L(x)$. Thus our induction hypothesis implies that $\mu_L(z,x)=\mu_{\boldsymbol{\lsem}z,x\boldsymbol{\rsem}}(z,x)$ can be nonzero only if $z\in\boldsymbol{\lsem}\xi_{\boldsymbol{\lsem}z,x\boldsymbol{\rsem}}(x),x\boldsymbol{\rsem}\subset\boldsymbol{\lsem}\xi_L(x),x\boldsymbol{\rsem}$. Thus the claim is true for all $z$ with $\eta(z)<m$.
 
Suppose then that $z\not\in\boldsymbol{\lsem}\xi_L(x),x\boldsymbol{\rsem}$ and $\eta(z)=m$. We aim to prove that $\mu_L(z,x)=0$ by applying the formula
\[
\mu_L(z,x)=-\sum_{z\prec v\preceq x}\mu_L(v,x).
\]
Since in this case $z$ is a lower bound for all the elements $y_i$, we must have $z\preceq \xi_L(x)$. When calculating the value $\mu_L(z,x)$ we may omit all elements $v\succ z$ such that $\mu_L(v,x)=0$. If $\eta(v)<m$, then by the work done above we know that all the elements $v$ with nonzero Möbius function value are located on the interval $\boldsymbol{\lsem}\xi_{\boldsymbol{\lsem}v,x\boldsymbol{\rsem}}(x),x\boldsymbol{\rsem}\subseteq\boldsymbol{\lsem}\xi_L(x),x\boldsymbol{\rsem}$. All the remaining elements $v\succ z$ have $\eta(v)=m$, and therefore $v\in\boldsymbol{\rsem}z,\xi_L(x)\boldsymbol{\rsem}$. Thus
\[
\mu_L(z,x)=-\left(\sum_{z\prec v\preceq x}\mu_L(v,x)\right)=-\left(\sum_{z\prec v\prec \xi_L(x)}\mu_L(v,x)\right)-\left(\sum_{\xi_L(x)\preceq v\preceq x}\mu_L(v,x)\right).
\]
Next we use induction on the size of the interval $\boldsymbol{\lsem}z,\xi_L(x)\boldsymbol{\lsem}$ and show that $\mu_L(z,x)=0$. If there is only one element on this interval, then
\[
\mu_L(z,x)=-\left(\sum_{z\prec v\preceq x}\mu_L(v,x)\right)=-\left(\sum_{\xi_L(x)\preceq v\preceq x}\mu_L(v,x)\right)=\delta_L(\xi_L(x),x)=0,
\]
since $\xi_L(x)\neq x$. Our induction hypothesis is that if there are less than $k(\geq2)$ elements on the interval $\boldsymbol{\lsem}v,\xi_L(x)\boldsymbol{\lsem}$, then $\mu_L(v,x)=0$. Suppose that there are $k$ elements on the interval $\boldsymbol{\lsem}z,\xi_L(x)\boldsymbol{\lsem}$. Now
\[
\mu_L(z,x)=-\left(\sum_{z\prec v\prec \xi_{\boldsymbol{\lsem}z,x\boldsymbol{\rsem}}(x)}\overbrace{\mu_L(v,x)}^{\substack{=0\text{ by induction}\\\text{hypothesis}}}\right)-\left(\overbrace{\sum_{\xi_{\boldsymbol{\lsem}z,x\boldsymbol{\rsem}}(x)\preceq u\preceq x}\mu_L(u,x)}^{=\delta_L(\xi_L(x),x)=0}\right)=-0-0=0.
\]
Thus our proof is complete.
\end{proof}

\begin{remark}\label{re:möbius0}
It is also possible to prove a stronger version of Lemma \ref{Möbius0}. Suppose that $\mu_L(z,x)\neq0$. By using the same notations as in the proof of Lemma \ref{Möbius0} we trivially have $z\preceq \xi_{\boldsymbol{\lsem}z,x\boldsymbol{\rsem}}(x)$. On the other hand, Lemma \ref{Möbius0} implies that $\xi_{\boldsymbol{\lsem}z,x\boldsymbol{\rsem}}(x)\preceq z\preceq x$. Thus we must have $\xi_{\boldsymbol{\lsem}z,x\boldsymbol{\rsem}}(x)=z$. In other words, if $\mu_L(z,x)\neq0$, then $z$ has to be the meet of all those elements $y_i$, which are covered by $x$ and located on the interval $\boldsymbol{\lsem}z,x\boldsymbol{\rsem}$. However, for our purposes the original formation is mostly sufficient, but in the proof of Theorem \ref{th:yksikäsitteisyyslause} this stronger version is needed as well.
\end{remark}

\section{Singularity of the usual LCM matrices}
\setcounter{subfigure}{0}
It has been known for long that the smallest GCD closed set $S$, for which the LCM matrix $[S]$ is singular, has $8$ elements. However, the uniqueness of the structure of such set has not been considered earlier. The next theorem answers to this question.

\begin{theorem}\label{th:yksikäsitteisyyslause}
If $S$ is a GCD closed set with $8$ elements and the LCM matrix $[S]=[[x_i,x_j]]$ is singular, then the semilattice $(S,|)$ always belongs to the class $8_{\rm J}$ in Figure \ref{fig:luokat}.
\end{theorem}

\begin{proof}
Suppose that $S$ is a GCD closed set with $8$ elements and its LCM matrix $[S]$ is singular. Thus $S_1=\{x_1\}, S_2=\{x_1,x_2\},\ldots,S_8=S$. Since $S_1,\ldots,S_7$ are all meet semilattices with less than $8$ elements and all LCM matrices are invertible up to the size $7\times7$, we know that $\Psi_{S_i,\frac{1}{N}}(x_i)=\Psi_{S,\frac{1}{N}}(x_i)\neq0$ for all $i=1,\ldots,7$. Since the matrix $[S]$ is yet singular, we must have $\Psi_{S,\frac{1}{N}}(x_8)=0$.

Next we should note that the last added element $x_8$ must cover at least three elements. Otherwise Remark \ref{re:möbius0} would imply that the set of all elements $x_i\in S$ with $\mu_S(x_i,x_8)\neq0$ belongs to either of the classes ${\cal S}_{1,2}$ or ${\cal S}_{2,4}$ in Figure \ref{fig:aputulos}. In the first case we have $\Psi_{S,\frac{1}{N}}(x_8)=\frac{1}{x_8}-\frac{1}{x_k}<0$, where $x_k$ is the element covered by $x_8$. In the second case $\Psi_{S,\frac{1}{N}}(x_8)>0$ by Proposition \ref{aputulos}. Furthermore, from this we deduce that in the Hasse diagram of $(S,|)$ every maximal element has to cover at least three elements. If this is not the case and there is a maximal element that covers at most two elements, then the set $S$ can be constructed so that $x_8$ is this element. As above we obtain that $\Psi_{S,\frac{1}{N}}(x_8)\neq0$, which is a contradiction.

There are $1078$ meet semilattices with $8$ elements, but the condition that every maximal element need to cover at least three elements reduces the number of possibilities to $84$ (Remark \ref{re:sage} contains the details on how the desired list of meet semilattices is obtained). By taking into account the possible zeros of the Möbius function $\mu_S$ we are able to rule out even more structures, namely those for which there exists $x_i\in S$ such that $\mu_S(x_i,x_8)=0$, $x_i$ covers at most one element and is covered by exactly one. Suppose for a contradiction that there exists such element $x_i$ in $S$. Then $S\setminus\{x_i\}$ is a meet semilattice with 7 elements (the ordering of $S\setminus\{x_i\}$ is induced by the ordering of $S$), $\mu_{S\setminus\{x_i\}}(x_k,x_8)=\mu_S(x_k,x_8)$ for all $x_k\in S\setminus\{x_i\}$ and therefore
\[
\Psi_{S,\frac{1}{N}}(x_8)=\Psi_{S\setminus\{x_i\},\frac{1}{N}}(x_8)\neq0.
\]
Again this means that the matrix $[S]$ is invertible, which is a contradiction. Thus $S$ cannot contain this type of element $x_i$. This leaves us with the ten possible structures $8_{\rm A},\ldots,8_{\rm J}$ presented in Figure \ref{fig:luokat}. By the work by Hong \cite{Ho} we already know that $S$ may belong to the class $8_{\rm J}$. We only need to show that $S$ cannot be any of the remaining types $8_{\rm A},8_{\rm B},\ldots,8_{\rm I}$. It suffices to prove that $\Psi_{S,\frac{1}{N}}(x_8)\neq0$ whenever $(S,|)$ belongs to any of these classes. From Proposition \ref{aputulos} we obtain directly that $\Psi_{S,\frac{1}{N}}(x_8)>0$ when $S\in 8_{\rm I}$. In order to revert also the remaining cases to this same proposition we first need to divide the set $S$ into suitable blocks. Figure \ref{fig:valinnat} shows the indexing of the elements of $S$ in each case.
\begin{figure}[ht!]
\centering
\subfigure[$8_{\rm A}$]
{{\scalefont{0.5}
\begin{tikzpicture}[scale=0.65]
\draw (2.3,-0.5)--(1,0)--(0.3,1)--(1.43,1.93);
\draw (1,0)--(1,1)--(1.43,1.93);
\draw (2.3,-0.5)--(3.4,1)--(1.6,1.96);
\draw (1,0)--(1.8,1)--(1.57,1.93);
\draw (2.3,-0.5)--(2.6,1)--(1.57,1.93);
\draw [fill] (2.3,-0.5) circle [radius=0.1];
\draw [fill] (1,0) circle [radius=0.1];
\draw [fill] (0.3,1) circle [radius=0.1];
\draw [fill] (1,1) circle [radius=0.1];
\draw [fill] (1.8,1) circle [radius=0.1];
\draw [fill] (2.6,1) circle [radius=0.1];
\draw (1.5,2) circle [radius=0.1];
\draw [fill] (1.8,1) circle [radius=0.1];
\draw [fill] (3.4,1) circle [radius=0.1];
\node [right] at (2.3,-0.5) {$2$};
\node [right] at (0.9,-0.3) {$2$};
\node [right] at (0.3,1) {\textrm{-}$1$};
\node [right] at (1.8,1) {\textrm{-}$1$};
\node [right] at (2.6,1) {\textrm{-}$1$};
\node [right] at (1,1) {\textrm{-}$1$};
\node [right] at (3.4,1) {\textrm{-}$1$};
\node [right] at (1.5,2.1) {$1$};
\end{tikzpicture}}
}
\hspace{-2mm}
\subfigure[$8_{\rm B}$]
{{\scalefont{0.5}
\begin{tikzpicture}[scale=0.65]
\draw (2.3,-0.5)--(1,0)--(0.3,1)--(1.43,1.93);
\draw (1,0)--(1,1)--(1.43,1.93);
\draw (1,0)--(1.8,1)--(1.57,1.93);
\draw (2.3,-0.5)--(2.6,1)--(1.57,1.93);
\draw (2.45,0.25)--(1.8,1);
\draw [fill] (2.3,-0.5) circle [radius=0.1];
\draw [fill] (1,0) circle [radius=0.1];
\draw [fill] (0.3,1) circle [radius=0.1];
\draw [fill] (1,1) circle [radius=0.1];
\draw [fill] (1.8,1) circle [radius=0.1];
\draw [fill] (2.6,1) circle [radius=0.1];
\draw [fill] (2.45,0.25) circle [radius=0.1];
\draw (1.5,2) circle [radius=0.1];
\draw [fill] (1.8,1) circle [radius=0.1];
\node [right] at (2.3,-0.5) {$0$};
\node [right] at (2.45,0.25) {$1$};
\node [right] at (0.9,-0.3) {$2$};
\node [right] at (0.3,1) {\textrm{-}$1$};
\node [right] at (1.8,1) {\textrm{-}$1$};
\node [right] at (2.6,1) {\textrm{-}$1$};
\node [right] at (1,1) {\textrm{-}$1$};
\node [right] at (1.5,2) {$1$};
\end{tikzpicture}}
}
\hspace{-2mm}
\subfigure[$8_{\rm C}$]
{{\scalefont{0.5}
\begin{tikzpicture}[scale=0.65]
\draw (2,0)--(1,-0.5)--(0,0)--(0,1)--(0.92,1.93);
\draw (2,0)--(1,0.5);
\draw (0,0)--(1,0.5)--(1,1.9);
\draw (2,0)--(2,1)--(1.08,1.93);
\draw [fill] (2,0) circle [radius=0.1];
\draw [fill] (1,-0.5) circle [radius=0.1];
\draw [fill] (0,0) circle [radius=0.1];
\draw [fill] (0,1) circle [radius=0.1];
\draw [fill] (1,1.25) circle [radius=0.1];
\draw [fill] (1,0.5) circle [radius=0.1];
\draw (1,2) circle [radius=0.1];
\draw [fill] (2,1) circle [radius=0.1];
\node [right] at (2,0) {$1$};
\node [right] at (1,1.25) {\textrm{-}$1$};
\node [right] at (1,-0.6) {$0$};
\node [right] at (0.2,0) {$1$};
\node [right] at (0,1) {\textrm{-}$1$};
\node [right] at (1.05,0.57) {$0$};
\node [right] at (2,1) {\textrm{-}$1$};
\node [right] at (1,2) {$1$};
\end{tikzpicture}}
}
\hspace{-2mm}
\subfigure[$8_{\rm D}$]
{{\scalefont{0.5}
\begin{tikzpicture}[scale=0.65]
\draw (0,2)--(-0.5,3)--(-0.5,4)--(-0.06,4.93);
\draw (0,2)--(0.5,3)--(0.5,4)--(0.06,4.93);
\draw (-0.5,3)--(-1.25,4)--(-0.06,4.93);
\draw (0.5,3)--(1.25,4)--(0.06,4.93);
\draw [fill] (0.5,4) circle [radius=0.1];
\draw (0,5) circle [radius=0.1];
\draw [fill] (0,2) circle [radius=0.1];
\draw [fill] (1.25,4) circle [radius=0.1];
\draw [fill] (-1.25,4) circle [radius=0.1];
\draw [fill] (-0.5,3) circle [radius=0.1];
\draw [fill] (0.5,3) circle [radius=0.1];
\draw [fill] (-0.5,4) circle [radius=0.1];
\node [right] at (-0.5,4) {\textrm{-}$1$};
\node [right] at (0,2) {$1$};
\node [right] at (-0.5,3) {$1$};
\node [right] at (0.5,3) {$1$};
\node [right] at (0.5,4) {\textrm{-}$1$};
\node [right] at (0,5) {$1$};
\node [right] at (-1.25,4) {\textrm{-}$1$};
\node [right] at (1.25,4) {\textrm{-}$1$};
\end{tikzpicture}}
}
\hspace{-2mm}
\subfigure[$8_{\rm E}$]
{{\scalefont{0.5}
\begin{tikzpicture}[scale=0.65]
\draw (2.3,-0.5)--(1,0)--(0.3,1)--(1.43,1.93);
\draw (1,0)--(1,1)--(1.43,1.93);
\draw (1,0)--(1.8,1)--(1.57,1.93);
\draw (1,0)--(-0.5,1)--(1.42,1.96);
\draw (2.3,-0.5)--(2.6,1)--(1.57,1.93);
\draw [fill] (2.3,-0.5) circle [radius=0.1];
\draw [fill] (1,0) circle [radius=0.1];
\draw [fill] (0.3,1) circle [radius=0.1];
\draw [fill] (1,1) circle [radius=0.1];
\draw [fill] (1.8,1) circle [radius=0.1];
\draw [fill] (2.6,1) circle [radius=0.1];
\draw [fill] (-0.5,1) circle [radius=0.1];
\draw (1.5,2) circle [radius=0.1];
\draw [fill] (1.8,1) circle [radius=0.1];
\node [right] at (2.3,-0.5) {$1$};
\node [right] at (0.9,-0.3) {$3$};
\node [right] at (0.3,1) {\textrm{-}$1$};
\node [right] at (1.8,1) {\textrm{-}$1$};
\node [right] at (2.6,1) {\textrm{-}$1$};
\node [right] at (1,1) {\textrm{-}$1$};
\node [right] at (-0.5,1) {\textrm{-}$1$};
\node [right] at (1.5,2) {$1$};
\end{tikzpicture}}
}
\hspace{-4mm}
\subfigure[$8_{\rm F}$]
{{\scalefont{0.5}
\begin{tikzpicture}[scale=0.65]
\draw (2,0.2)--(1,-0.5)--(0,0.2)--(0,1)--(0.92,1.93);
\draw (2,0.2)--(1,1);
\draw (0,0.2)--(1,1)--(1,1.9);
\draw (1,-0.5)--(3.3,0.5)--(1.1,1.95);
\draw (2,0.2)--(2,1)--(1.08,1.93);
\draw [fill] (2,0.2) circle [radius=0.1];
\draw [fill] (1,-0.5) circle [radius=0.1];
\draw [fill] (0,0.2) circle [radius=0.1];
\draw [fill] (0,1) circle [radius=0.1];
\draw [fill] (1,1) circle [radius=0.1];
\draw [fill] (3.3,0.5) circle [radius=0.1];
\draw (1,2) circle [radius=0.1];
\draw [fill] (2,1) circle [radius=0.1];
\node [right] at (2,0.25) {$1$};
\node [right] at (1,-0.7) {$1$};
\node [right] at (0.05,0.2) {$1$};
\node [right] at (0,1) {\textrm{-}$1$};
\node [right] at (1,1) {\textrm{-}$1$};
\node [right] at (2,0.9) {\textrm{-}$1$};
\node [right] at (1,2.1) {$1$};
\node [right] at (3.3,0.5) {\textrm{-}$1$};
\end{tikzpicture}}
}
\hspace{-4mm}
\subfigure[$8_{\rm G}$]
{{\scalefont{0.5}
\begin{tikzpicture}[scale=0.65]
\draw (2.5,-0.5)--(1.5,0)--(0.5,1)--(1.43,1.93);
\draw (1.5,0)--(1.2,1)--(1.43,1.93);
\draw (2.5,-0.5)--(2,1)--(1.57,1.93);
\draw (2.5,-0.5)--(2.8,1)--(1.57,1.95);
\draw (2.5,-0.5)--(3.6,1)--(1.58,1.96);
\draw [fill] (2.5,-0.5) circle [radius=0.1];
\draw [fill] (1.5,0) circle [radius=0.1];
\draw [fill] (0.5,1) circle [radius=0.1];
\draw [fill] (1.2,1) circle [radius=0.1];
\draw [fill] (2,1) circle [radius=0.1];
\draw [fill] (3.6,1) circle [radius=0.1];
\draw [fill] (2.8,1) circle [radius=0.1];
\draw (1.5,2) circle [radius=0.1];
\draw [fill] (2,1) circle [radius=0.1];
\node [right] at (2.5,-0.5) {$3$};
\node [right] at (1.5,0.05) {$1$};
\node [right] at (0.5,1) {\textrm{-}$1$};
\node [right] at (2,1) {\textrm{-}$1$};
\node [right] at (2.8,0.95) {\textrm{-}$1$};
\node [right] at (1.2,1) {\textrm{-}$1$};
\node [right] at (3.6,0.95) {\textrm{-}$1$};
\node [right] at (1.5,2.1) {$1$};
\end{tikzpicture}}
}
\hspace{-4mm}
\subfigure[$8_{\rm H}$]
{{\scalefont{0.5}
\begin{tikzpicture}[scale=0.65]
\draw (0,0)--(-1.5,1.5)--(-0.1,2.37);
\draw (-1,1)--(-0.5,1.5)--(-0.04,2.32);
\draw (-0.5,0.5)--(0.5,1.5)--(0.04,2.32);
\draw (0,0)--(1.5,1.5)--(0.1,2.37);
\draw (0,2.4) circle [radius=0.1];
\draw [fill] (0,0) circle [radius=0.1];
\draw [fill] (-0.5,0.5) circle [radius=0.1];
\draw [fill] (-1,1) circle [radius=0.1];
\draw [fill] (-1.5,1.5) circle [radius=0.1];
\draw [fill] (1.5,1.5) circle [radius=0.1];
\draw [fill] (-0.5,1.5) circle [radius=0.1];
\draw [fill] (0.5,1.5) circle [radius=0.1];
\node [right] at (0,2.55) {$1$};
\node [right] at (-1.5,1.5) {\textrm{-}$1$};
\node [right] at (-0.5,1.5) {\textrm{-}$1$};
\node [right] at (0.5,1.5) {\textrm{-}$1$};
\node [right] at (1.5,1.5) {\textrm{-}$1$};
\node [right] at (-0.45,0.5) {$1$};
\node [right] at (-0.95,1) {$1$};
\node [right] at (0.05,0) {$1$};
\end{tikzpicture}}
}
\hspace{-4mm}
\subfigure[$8_{\rm I}$]
{{\scalefont{0.5}
\begin{tikzpicture}[scale=0.65]
\draw (1,0)--(-1.75,1)--(0.92,1.92);
\draw (1,0)--(-0.75,1)--(0.92,1.92);
\draw (1,0)--(0.2,1)--(0.92,1.92);
\draw (1,0)--(1.8,1)--(1.08,1.92);
\draw (1,0)--(2.75,1)--(1.08,1.92);
\draw (1,0)--(3.75,1)--(1.08,1.92);
\draw [fill] (-1.75,1) circle [radius=0.1];
\draw [fill] (3.75,1) circle [radius=0.1];
\draw [fill] (0.2,1) circle [radius=0.1];
\draw [fill] (1.8,1) circle [radius=0.1];
\draw [fill] (1,0) circle [radius=0.1];
\draw [fill] (2.75,1) circle [radius=0.1];
\draw [fill] (-0.75,1) circle [radius=0.1];
\draw (1,2) circle [radius=0.1];
\node [right] at (1,-0.1) {$5$};
\node [right] at (0.2,1) {\textrm{-}$1$};
\node [right] at (-0.7,1) {\textrm{-}$1$};
\node [right] at (2.75,1) {\textrm{-}$1$};
\node [right] at (1.8,1) {\textrm{-}$1$};
\node [right] at (1,2.05) {$1$};
\node [right] at (-1.75,1) {\ \textrm{-}$1$};
\node [right] at (3.75,1) {\textrm{-}$1$};
\end{tikzpicture}}
}
\hspace{-4mm}
\subfigure[$8_{\rm J}$]
{{\scalefont{0.5}
\begin{tikzpicture}[scale=0.65]
\draw (1,0)--(0,1)--(0,2)--(0.92,2.92);
\draw (1,0)--(1,1)--(0,2);
\draw (0,1)--(1,2)--(1,2.9);
\draw (1,0)--(2,1)--(2,2)--(1.08,2.92);
\draw (1,1)--(2,2);
\draw (2,1)--(1,2);
\draw [fill] (0,1) circle [radius=0.1];
\draw [fill] (1,1) circle [radius=0.1];
\draw [fill] (2,1) circle [radius=0.1];
\draw [fill] (1,0) circle [radius=0.1];
\draw [fill] (0,2) circle [radius=0.1];
\draw [fill] (1,2) circle [radius=0.1];
\draw [fill] (2,2) circle [radius=0.1];
\draw (1,3) circle [radius=0.1];
\node [right] at (1,0) {\textrm{-}$1$};
\node [right] at (0,1) {$1$};
\node [right] at (1,1) {$1$};
\node [right] at (2,1) {$1$};
\node [right] at (0,2) {\textrm{-}$1$};
\node [right] at (1,2) {\textrm{-}$1$};
\node [right] at (2,2) {\textrm{-}$1$};
\node [right] at (1,3) {$1$};
\end{tikzpicture}}}
\caption{The Hasse-diagrams of the meet semilattices in the proof of Theorem \ref{th:yksikäsitteisyyslause}. For every semilattice the number next to each element is the value $\mu_S(x_i,x_8)$, where $x_8$ is the last added element and is denoted by the white dot.}\label{fig:luokat}
\end{figure}
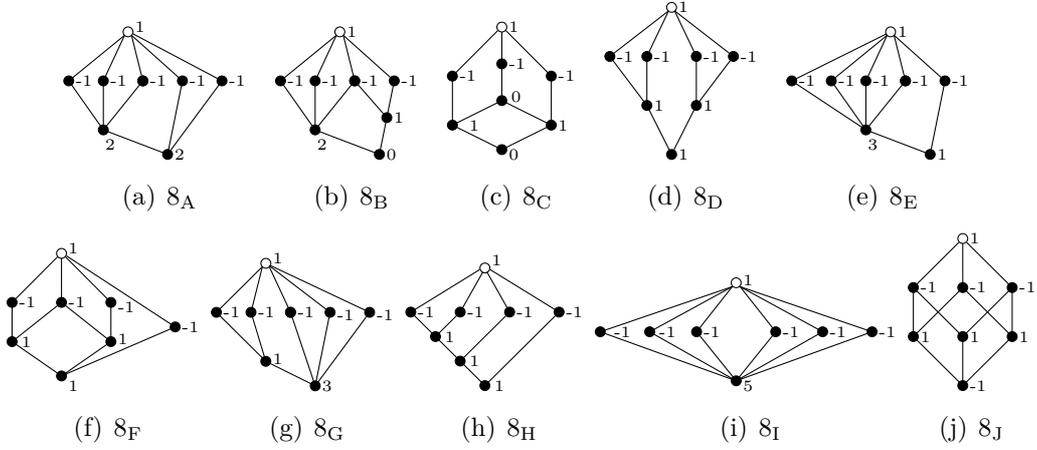

\begin{figure}[ht!]
\centering
\subfigure[$8_{\rm A}$]
{{\scalefont{0.8}
\begin{tikzpicture}[scale=0.75]
\draw (2.3,-0.5)--(1,0)--(0.3,1)--(1.43,1.93);
\draw (1,0)--(1,1)--(1.43,1.93);
\draw (2.3,-0.5)--(3.4,1)--(1.6,1.96);
\draw (1,0)--(1.8,1)--(1.57,1.93);
\draw (2.3,-0.5)--(2.6,1)--(1.57,1.93);
\draw [fill] (2.3,-0.5) circle [radius=0.1];
\draw [fill] (1,0) circle [radius=0.1];
\draw [fill] (0.3,1) circle [radius=0.1];
\draw [fill] (1,1) circle [radius=0.1];
\draw [fill] (1.8,1) circle [radius=0.1];
\draw [fill] (2.6,1) circle [radius=0.1];
\draw [fill] (1.5,2) circle [radius=0.1];
\draw [fill] (1.8,1) circle [radius=0.1];
\draw [fill] (3.4,1) circle [radius=0.1];
\node [right] at (2.3,-0.5) {$x_1$};
\node [right] at (0.9,-0.3) {$x_2$};
\node [right] at (0.3,1) {$x_3$};
\node [right] at (1.8,1) {$x_5$};
\node [right] at (2.6,1) {$x_6$};
\node [right] at (1,1) {$x_4$};
\node [right] at (3.4,1) {$x_7$};
\node [right] at (1.5,2.1) {$x_8$};
\end{tikzpicture}}
}
\subfigure[$8_{\rm B}$]
{{\scalefont{0.8}
\begin{tikzpicture}[scale=0.75]
\draw (2.3,-0.5)--(1,0)--(0.3,1)--(1.43,1.93);
\draw (1,0)--(1,1)--(1.43,1.93);
\draw (1,0)--(1.8,1)--(1.57,1.93);
\draw (2.3,-0.5)--(2.6,1)--(1.57,1.93);
\draw (2.45,0.25)--(1.8,1);
\draw [fill] (2.3,-0.5) circle [radius=0.1];
\draw [fill] (1,0) circle [radius=0.1];
\draw [fill] (0.3,1) circle [radius=0.1];
\draw [fill] (1,1) circle [radius=0.1];
\draw [fill] (1.8,1) circle [radius=0.1];
\draw [fill] (2.6,1) circle [radius=0.1];
\draw [fill] (2.45,0.25) circle [radius=0.1];
\draw [fill] (1.5,2) circle [radius=0.1];
\draw [fill] (1.8,1) circle [radius=0.1];
\node [right] at (2.3,-0.5) {$x_1$};
\node [right] at (2.45,0.25) {$x_3$};
\node [right] at (0.9,-0.3) {$x_2$};
\node [right] at (0.3,1) {$x_4$};
\node [right] at (1.8,1) {$x_6$};
\node [right] at (2.6,1) {$x_7$};
\node [right] at (1,1) {$x_5$};
\node [right] at (1.5,2) {$x_8$};
\end{tikzpicture}}
}
\subfigure[$8_{\rm C}$]
{{\scalefont{0.8}
\begin{tikzpicture}[scale=0.75]
\draw (2,0)--(1,-0.5)--(0,0)--(0,1)--(0.92,1.93);
\draw (2,0)--(1,0.5);
\draw (0,0)--(1,0.5)--(1,1.9);
\draw (2,0)--(2,1)--(1.08,1.93);
\draw [fill] (2,0) circle [radius=0.1];
\draw [fill] (1,-0.5) circle [radius=0.1];
\draw [fill] (0,0) circle [radius=0.1];
\draw [fill] (0,1) circle [radius=0.1];
\draw [fill] (1,1.25) circle [radius=0.1];
\draw [fill] (1,0.5) circle [radius=0.1];
\draw [fill] (1,2) circle [radius=0.1];
\draw [fill] (2,1) circle [radius=0.1];
\node [right] at (2,0) {$x_3$};
\node [right] at (1,1.25) {$x_6$};
\node [right] at (1,-0.6) {$x_1$};
\node [right] at (0.15,-0.03) {$x_2$};
\node [right] at (0,1) {$x_4$};
\node [right] at (1.05,0.57) {$x_5$};
\node [right] at (2,1) {$x_7$};
\node [right] at (1,2) {$x_8$};
\end{tikzpicture}}
}
\subfigure[$8_{\rm D}$]
{{\scalefont{0.8}
\begin{tikzpicture}[scale=0.75]
\draw (0,2)--(-0.5,3)--(-0.5,4)--(-0.06,4.93);
\draw (0,2)--(0.5,3)--(0.5,4)--(0.06,4.93);
\draw (-0.5,3)--(-1.25,4)--(-0.06,4.93);
\draw (0.5,3)--(1.25,4)--(0.06,4.93);
\draw [fill] (0.5,4) circle [radius=0.1];
\draw [fill] (0,5) circle [radius=0.1];
\draw [fill] (0,2) circle [radius=0.1];
\draw [fill] (1.25,4) circle [radius=0.1];
\draw [fill] (-1.25,4) circle [radius=0.1];
\draw [fill] (-0.5,3) circle [radius=0.1];
\draw [fill] (0.5,3) circle [radius=0.1];
\draw [fill] (-0.5,4) circle [radius=0.1];
\node [right] at (-0.5,4) {$x_5$};
\node [right] at (0,2) {$x_1$};
\node [right] at (-0.5,3) {$x_2$};
\node [right] at (0.5,3) {$x_3$};
\node [right] at (0.5,4) {$x_6$};
\node [right] at (0,5) {$x_8$};
\node [right] at (-1.25,4) {$x_4$};
\node [right] at (1.25,4) {$x_7$};
\end{tikzpicture}}
}
\subfigure[$8_{\rm E}$]
{{\scalefont{0.8}
\begin{tikzpicture}[scale=0.75]
\draw (2.3,-0.5)--(1,0)--(0.3,1)--(1.43,1.93);
\draw (1,0)--(1,1)--(1.43,1.93);
\draw (1,0)--(1.8,1)--(1.57,1.93);
\draw (1,0)--(-0.5,1)--(1.42,1.96);
\draw (2.3,-0.5)--(2.6,1)--(1.57,1.93);
\draw [fill] (2.3,-0.5) circle [radius=0.1];
\draw [fill] (1,0) circle [radius=0.1];
\draw [fill] (0.3,1) circle [radius=0.1];
\draw [fill] (1,1) circle [radius=0.1];
\draw [fill] (1.8,1) circle [radius=0.1];
\draw [fill] (2.6,1) circle [radius=0.1];
\draw [fill] (-0.5,1) circle [radius=0.1];
\draw [fill] (1.5,2) circle [radius=0.1];
\draw [fill] (1.8,1) circle [radius=0.1];
\node [right] at (2.3,-0.5) {$x_1$};
\node [right] at (0.9,-0.3) {$x_2$};
\node [right] at (0.3,1) {$x_4$};
\node [right] at (1.8,1) {$x_6$};
\node [right] at (2.6,1) {$x_7$};
\node [right] at (1,1) {$x_5$};
\node [right] at (-0.44,0.96) {$x_3$};
\node [right] at (1.5,2) {$x_8$};
\end{tikzpicture}}
}
\subfigure[$8_{\rm F}$]
{{\scalefont{0.8}
\begin{tikzpicture}[scale=0.75]
\draw (2,0.2)--(1,-0.5)--(0,0.2)--(0,1)--(0.92,1.93);
\draw (2,0.2)--(1,1);
\draw (0,0.2)--(1,1)--(1,1.9);
\draw (1,-0.5)--(3.3,0.5)--(1.1,1.95);
\draw (2,0.2)--(2,1)--(1.08,1.93);
\draw [fill] (2,0.2) circle [radius=0.1];
\draw [fill] (1,-0.5) circle [radius=0.1];
\draw [fill] (0,0.2) circle [radius=0.1];
\draw [fill] (0,1) circle [radius=0.1];
\draw [fill] (1,1) circle [radius=0.1];
\draw [fill] (3.3,0.5) circle [radius=0.1];
\draw [fill] (1,2) circle [radius=0.1];
\draw [fill] (2,1) circle [radius=0.1];
\node [right] at (2,0.32) {$x_3$};
\node [right] at (1,-0.7) {$x_1$};
\node [right] at (0.05,0.2) {$x_2$};
\node [right] at (0,1) {$x_5$};
\node [right] at (1,1) {$x_6$};
\node [right] at (2,0.85) {$x_7$};
\node [right] at (1,2.1) {$x_8$};
\node [right] at (3.3,0.5) {$x_4$};
\end{tikzpicture}}
}
\subfigure[$8_{\rm G}$]
{{\scalefont{0.8}
\begin{tikzpicture}[scale=0.75]
\draw (2.5,-0.5)--(1.5,0)--(0.5,1)--(1.43,1.93);
\draw (1.5,0)--(1.2,1)--(1.43,1.93);
\draw (2.5,-0.5)--(2,1)--(1.57,1.93);
\draw (2.5,-0.5)--(2.8,1)--(1.57,1.95);
\draw (2.5,-0.5)--(3.6,1)--(1.58,1.96);
\draw [fill] (2.5,-0.5) circle [radius=0.1];
\draw [fill] (1.5,0) circle [radius=0.1];
\draw [fill] (0.5,1) circle [radius=0.1];
\draw [fill] (1.2,1) circle [radius=0.1];
\draw [fill] (2,1) circle [radius=0.1];
\draw [fill] (3.6,1) circle [radius=0.1];
\draw [fill] (2.8,1) circle [radius=0.1];
\draw [fill] (1.5,2) circle [radius=0.1];
\draw [fill] (2,1) circle [radius=0.1];
\node [right] at (2.5,-0.5) {$x_1$};
\node [right] at (1.5,0.05) {$x_2$};
\node [right] at (0.5,1) {$x_3$};
\node [right] at (2,1) {$x_5$};
\node [right] at (2.8,0.95) {$x_6$};
\node [right] at (1.2,1) {$x_4$};
\node [right] at (3.6,0.95) {$x_7$};
\node [right] at (1.5,2.1) {$x_8$};
\end{tikzpicture}}
}
\subfigure[$8_{\rm H}$]
{{\scalefont{0.8}
\begin{tikzpicture}[scale=0.75]
\draw (0,0)--(-1.5,1.5)--(-0.1,2.37);
\draw (-1,1)--(-0.5,1.5)--(-0.04,2.32);
\draw (-0.5,0.5)--(0.5,1.5)--(0.04,2.32);
\draw (0,0)--(1.5,1.5)--(0.1,2.37);
\draw [fill] (0,2.4) circle [radius=0.1];
\draw [fill] (0,0) circle [radius=0.1];
\draw [fill] (-0.5,0.5) circle [radius=0.1];
\draw [fill] (-1,1) circle [radius=0.1];
\draw [fill] (-1.5,1.5) circle [radius=0.1];
\draw [fill] (1.5,1.5) circle [radius=0.1];
\draw [fill] (-0.5,1.5) circle [radius=0.1];
\draw [fill] (0.5,1.5) circle [radius=0.1];
\node [right] at (0,2.55) {$x_8$};
\node [right] at (-1.45,1.47) {$x_4$};
\node [right] at (-0.5,1.5) {$x_5$};
\node [right] at (0.5,1.5) {$x_6$};
\node [right] at (1.5,1.5) {$x_7$};
\node [right] at (-0.45,0.5) {$x_2$};
\node [right] at (-0.95,1) {$x_3$};
\node [right] at (0.05,0) {$x_1$};
\end{tikzpicture}}
}
\caption{The numbering of elements of $S$ in the cases when $S$ belongs to classes $8_{\rm A},8_{\rm B},\ldots,8_{\rm H}$.}\label{fig:valinnat}
\end{figure}
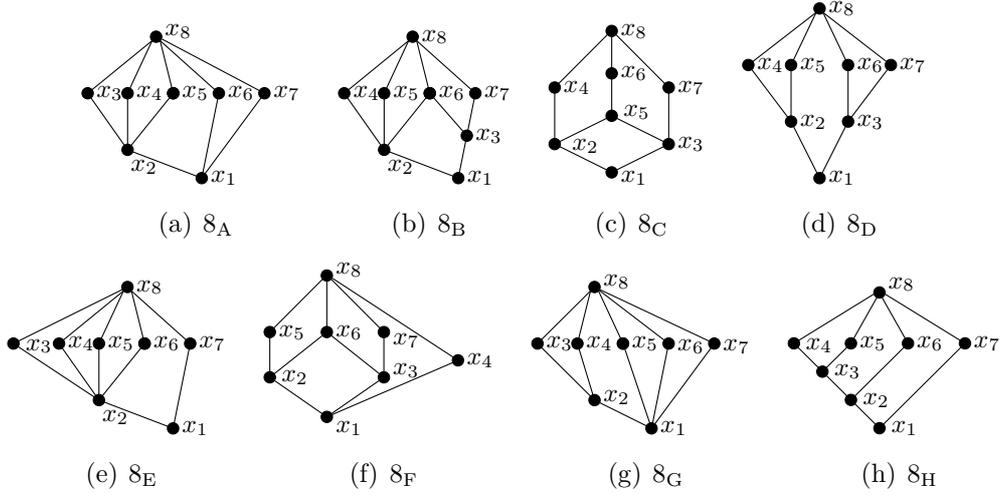

(i) Let $(S,|)\in 8_{\rm A}$. Then $\{x_2,x_3,x_4,x_5,x_8\}\in{\cal S}_{3,5}$, $\{x_1,x_6\},\{x_1,x_7\}\in{\cal S}_{1,2}$ and
\begin{align*}
\Psi_{S,\frac{1}{N}}(x_8)&=\frac{1}{x_8}-\frac{1}{x_7}-\frac{1}{x_6}-\frac{1}{x_5}-\frac{1}{x_4}-\frac{1}{x_3}+\frac{2}{x_2}+\frac{2}{x_1}\\
&=\underbrace{\left(\frac{1}{x_8}-\frac{1}{x_5}-\frac{1}{x_4}-\frac{1}{x_3}+\frac{2}{x_2}\right)}_{>0}+\underbrace{\left(\frac{1}{x_1}-\frac{1}{x_7}\right)}_{>0}+\underbrace{\left(\frac{1}{x_1}-\frac{1}{x_6}\right)}_{>0}>0.
\end{align*}

(ii) Let $(S,|)\in 8_{\rm B}$. Then $\{x_2,x_4,x_5,x_6,x_8\}\in{\cal S}_{3,5}$, $\{x_3,x_7\}\in{\cal S}_{1,2}$ and
\begin{align*}
\Psi_{S,\frac{1}{N}}(x_8)&=\frac{1}{x_8}-\frac{1}{x_7}-\frac{1}{x_6}-\frac{1}{x_5}-\frac{1}{x_4}+\frac{1}{x_3}+\frac{2}{x_2}\\
&=\underbrace{\left(\frac{1}{x_8}-\frac{1}{x_6}-\frac{1}{x_5}-\frac{1}{x_4}+\frac{2}{x_2}\right)}_{>0}+\underbrace{\left(\frac{1}{x_3}-\frac{1}{x_7}\right)}_{>0}>0.
\end{align*}

(iii) Let $(S,|)\in 8_{\rm C}$. Then $\{x_2,x_4,x_6,x_8\}\in{\cal S}_{2,4}$, $\{x_3,x_7\}\in{\cal S}_{1,2}$ and
\begin{align*}
\Psi_{S,\frac{1}{N}}(x_8)&=\frac{1}{x_8}-\frac{1}{x_7}-\frac{1}{x_6}-\frac{1}{x_4}+\frac{1}{x_3}+\frac{1}{x_2}\\
&=\underbrace{\left(\frac{1}{x_8}-\frac{1}{x_6}-\frac{1}{x_4}+\frac{1}{x_2}\right)}_{>0}+\underbrace{\left(\frac{1}{x_3}-\frac{1}{x_7}\right)}_{>0}>0.
\end{align*}

(iv) Let $(S,|)\in 8_{\rm D}$. Then $\{x_2,x_4,x_5,x_8\}\in{\cal S}_{2,4}$, $\{x_3,x_6\},\{x_1,x_7\}\in{\cal S}_{1,2}$ and
\begin{align*}
\Psi_{S,\frac{1}{N}}(x_8)&=\frac{1}{x_8}-\frac{1}{x_7}-\frac{1}{x_6}-\frac{1}{x_5}-\frac{1}{x_4}+\frac{1}{x_3}+\frac{1}{x_2}+\frac{1}{x_1}\\
&=\underbrace{\left(\frac{1}{x_8}-\frac{1}{x_5}-\frac{1}{x_4}+\frac{1}{x_2}\right)}_{>0}+\underbrace{\left(\frac{1}{x_3}-\frac{1}{x_6}\right)}_{>0}+\underbrace{\left(\frac{1}{x_1}-\frac{1}{x_7}\right)}_{>0}>0.
\end{align*}

(v) Let $(S,|)\in 8_{\rm E}$. Then $\{x_2,x_3,x_4,x_5,x_6,x_8\}\in{\cal S}_{4,6}$, $\{x_1,x_7\}\in{\cal S}_{1,2}$ and
\begin{align*}
\Psi_{S,\frac{1}{N}}(x_8)&=\frac{1}{x_8}-\frac{1}{x_7}-\frac{1}{x_6}-\frac{1}{x_5}-\frac{1}{x_4}-\frac{1}{x_3}+\frac{3}{x_2}+\frac{1}{x_1}\\
&=\underbrace{\left(\frac{1}{x_8}-\frac{1}{x_6}-\frac{1}{x_5}-\frac{1}{x_4}-\frac{1}{x_3}+\frac{3}{x_2}\right)}_{>0}+\underbrace{\left(\frac{1}{x_1}-\frac{1}{x_7}\right)}_{>0}>0.
\end{align*}

(vi) Let $(S,|)\in 8_{\rm F}$. Then $\{x_2,x_5,x_6,x_8\}\in{\cal S}_{2,4}$, $\{x_3,x_7\},\{x_1,x_4\}\in{\cal S}_{1,2}$ and
\begin{align*}
\Psi_{S,\frac{1}{N}}(x_8)&=\frac{1}{x_8}-\frac{1}{x_7}-\frac{1}{x_6}-\frac{1}{x_5}-\frac{1}{x_4}+\frac{1}{x_3}+\frac{1}{x_2}+\frac{1}{x_1}\\
&=\underbrace{\left(\frac{1}{x_8}-\frac{1}{x_6}-\frac{1}{x_5}+\frac{1}{x_2}\right)}_{>0}+\underbrace{\left(\frac{1}{x_3}-\frac{1}{x_7}\right)}_{>0}+\underbrace{\left(\frac{1}{x_1}-\frac{1}{x_4}\right)}_{>0}>0.
\end{align*}

(vii) Let $(S,|)\in 8_{\rm G}$. Then $\{x_1,x_5,x_6,x_7,x_8\}\in{\cal S}_{3,5}$, $\{x_2,x_3\},\{x_1,x_4\}\in{\cal S}_{1,2}$ and
\begin{align*}
\Psi_{S,\frac{1}{N}}(x_8)&=\frac{1}{x_8}-\frac{1}{x_7}-\frac{1}{x_6}-\frac{1}{x_5}-\frac{1}{x_4}-\frac{1}{x_3}+\frac{1}{x_2}+\frac{3}{x_1}\\
&=\underbrace{\left(\frac{1}{x_8}-\frac{1}{x_7}-\frac{1}{x_6}-\frac{1}{x_5}+\frac{2}{x_1}\right)}_{>0}+\underbrace{\left(\frac{1}{x_2}-\frac{1}{x_3}\right)}_{>0}+\underbrace{\left(\frac{1}{x_1}-\frac{1}{x_4}\right)}_{>0}>0.
\end{align*}

(viii) Let $(S,|)\in 8_{\rm H}$. Then $\{x_3,x_4,x_5,x_8\}\in{\cal S}_{2,4}$, $\{x_2,x_6\},\{x_1,x_7\}\in{\cal S}_{1,2}$ and
\begin{align*}
\Psi_{S,\frac{1}{N}}(x_8)&=\frac{1}{x_8}-\frac{1}{x_7}-\frac{1}{x_6}-\frac{1}{x_5}-\frac{1}{x_4}+\frac{1}{x_3}+\frac{1}{x_2}+\frac{1}{x_1}\\
&=\underbrace{\left(\frac{1}{x_8}-\frac{1}{x_5}-\frac{1}{x_4}+\frac{1}{x_3}\right)}_{>0}+\underbrace{\left(\frac{1}{x_2}-\frac{1}{x_6}\right)}_{>0}+\underbrace{\left(\frac{1}{x_1}-\frac{1}{x_7}\right)}_{>0}>0.
\end{align*}
Thus we have shown that $(S,|)$ must belong to class $8_{\rm J}$ and our proof is complete.
\end{proof}

\begin{remark}\label{re:sage}
Since there are $1078$ meet semilattices with $8$ elements, one of the main challenges in the proof of Theorem \ref{th:yksikäsitteisyyslause} is to find all suitable meet semilattices and rule out the rest. We used Sage 6.1.1 (see \cite{Stein}) in order to accomplish this. First we define the set of all meet semilattices with $8$ elements by using the command
\begin{verbatim}
L8=[p for p in Posets(8) if p.is_meet_semilattice()] .
\end{verbatim}
The command
\begin{verbatim}
L8=[l for l in L8 if not any (
 len(l.lower_covers(m))<3
 for  m in l.maximal_elements()  )] 
\end{verbatim}
then rules out all such semilattices in which some maximal element covers less than three elements. After that the command
\begin{verbatim}
L8=[l for l in L8 if not any (
 len(l.lower_covers(e))<=1 and
 len(l.upper_covers(e))==1 and
 l.mobius_function(e,7)==0
 for e in l.list()  )] 
\end{verbatim}
makes sure that in the remaining semilattices there are no elements $x_i$ such that $\mu_S(x_i,x_8)=0$, $x_i$ covers at most one other element and is covered by only one. Now the command 
\begin{verbatim}
for l in L8: l.dual().show()
\end{verbatim}
shows the Hasse diagrams of the meet semilattices $8_{\rm A},\ldots,8_{\rm J}$ in question.
\end{remark}

It is easy to see that if $S$ is an odd GCD closed set with at most $8$ elements, then the LCM matrix $[S]$ is always nonsingular (an odd set is a set whose all elements are odd). The only possibility to obtain a singular LCM matrix $[S]$ would be $(S,|)\in8_{\rm J}$, but this is impossible, since in this case
\begin{align*}
\Psi_{S,\frac{1}{N}}(x_8)&=\frac{1}{x_8}-\frac{1}{x_7}-\frac{1}{x_6}-\frac{1}{x_5}+\frac{1}{x_4}+\frac{1}{x_3}+\frac{1}{x_2}-\frac{1}{x_1}\\
&=\frac{1}{x_1}\underbrace{\left(-1+\underbrace{\frac{x_1}{x_2}}_{\leq\frac{1}{3}}+\underbrace{\frac{x_1}{x_3}}_{\leq\frac{1}{5}}+\underbrace{\frac{x_1}{x_4}}_{\leq\frac{1}{7}}\right)}_{<0}+\underbrace{\left(\frac{1}{x_8}-\frac{1}{x_7}-\frac{1}{x_6}-\frac{1}{x_5}\right)}_{<0}<0.
\end{align*}
Hong \cite{Ho3} took the idea of nonsigularity of LCM matrices of odd GCD closed sets even further by presenting the following conjecture:

\begin{conjecture}[\cite{Ho3}, Conjecture 4.4]\label{conj1}
The LCM matrix $[S]$ defined on any odd GCD closed set $S$ is nonsingular.
\end{conjecture}

However, this conjecture fails already when $n=9$.

\begin{theorem}\label{vastaesim}
Conjecture \ref{conj1} does not hold.
\end{theorem}

\begin{proof}
Let us consider the odd set
\begin{align*}
S&=\{1,3,5,7,195,291,1407,4025,1020180525\}\\
&=\{1,3,5,7,3\cdot5\cdot13,3\cdot97,3\cdot7\cdot67,5^2\cdot7\cdot23,3\cdot5^2\cdot7\cdot13\cdot23\cdot67\cdot97\}.
\end{align*}
After calculating the values of the Möbius function (see Figure \ref{fig:pariton}) we may apply \eqref{eq:psi} to obtain
\begin{align*}
&\Psi_{S,\frac{1}{N}}(1020180525)=\frac{1}{1020180525}-\frac{1}{4025}-\frac{1}{1407}-\frac{1}{291}-\frac{1}{195}+\frac{1}{7}+\frac{1}{5}+\frac{2}{3}-1\\
&=\frac{1}{1020180525}\big(1-253461-725075-3505775-5231695+145740075\\
&\hspace{5.5cm}+204036105+680120350-1020180525\big)=0,
\end{align*}
and thus it follows from Proposition \ref{psi-lause} that the matrix $[S]$ is singular.
\end{proof}

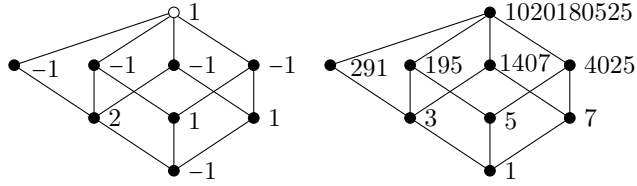
\begin{figure}[htb!]
\centering
\subfigure
{{\scalefont{0.8}
\begin{tikzpicture}[scale=0.7]
\draw (1,0)--(-0.5,1)--(-0.5,2)--(0.92,2.92);
\draw (1,0)--(1,1)--(-0.5,2);
\draw (-0.5,1)--(1,2)--(1,2.9);
\draw (1,0)--(2.5,1)--(2.5,2)--(1.08,2.92);
\draw (1,1)--(2.5,2);
\draw (-0.5,1)--(-2,2)--(0.91,2.98);
\draw (2.5,1)--(1,2);
\draw [fill] (-0.5,1) circle [radius=0.1];
\draw [fill] (1,1) circle [radius=0.1];
\draw [fill] (2.5,1) circle [radius=0.1];
\draw [fill] (1,0) circle [radius=0.1];
\draw [fill] (-0.5,2) circle [radius=0.1];
\draw [fill] (1,2) circle [radius=0.1];
\draw [fill] (2.5,2) circle [radius=0.1];
\draw [fill] (-2,2) circle [radius=0.1];
\draw (1,3) circle [radius=0.1];
\node [right] at (1.1,0) {$-1$};
\node [right] at (-0.4,1) {$2$};
\node [right] at (1.1,0.95) {$1$};
\node [right] at (2.6,1) {$1$};
\node [right] at (-0.4,2) {$-1$};
\node [right] at (1.1,2) {$-1$};
\node [right] at (2.6,2) {$-1$};
\node [right] at (-1.85,1.92) {$-1$};
\node [right] at (1.1,3) {$1$};
\end{tikzpicture}}
\subfigure
{{\scalefont{0.8}
\begin{tikzpicture}[scale=0.7]
\draw (1,0)--(-0.5,1)--(-0.5,2)--(0.92,2.92);
\draw (1,0)--(1,1)--(-0.5,2);
\draw (-0.5,1)--(1,2)--(1,2.9);
\draw (1,0)--(2.5,1)--(2.5,2)--(1.08,2.92);
\draw (1,1)--(2.5,2);
\draw (-0.5,1)--(-2,2)--(1,3);
\draw (2.5,1)--(1,2);
\draw [fill] (-0.5,1) circle [radius=0.1];
\draw [fill] (1,1) circle [radius=0.1];
\draw [fill] (2.5,1) circle [radius=0.1];
\draw [fill] (1,0) circle [radius=0.1];
\draw [fill] (-0.5,2) circle [radius=0.1];
\draw [fill] (1,2) circle [radius=0.1];
\draw [fill] (2.5,2) circle [radius=0.1];
\draw [fill] (-2,2) circle [radius=0.1];
\draw [fill] (1,3) circle [radius=0.1];
\node [right] at (1.1,0) {$1$};
\node [right] at (-0.4,1) {$3$};
\node [right] at (1.1,0.95) {$5$};
\node [right] at (2.6,1) {$7$};
\node [right] at (-0.4,2) {$195$};
\node [right] at (1,2.05) {$1407$};
\node [right] at (2.6,2) {$4025$};
\node [right] at (-1.8,1.95) {$291$};
\node [right] at (1.1,3) {$1020180525$};
\end{tikzpicture}}
}}
\caption{The Hasse diagram of the counterexample of Theorem \ref{vastaesim}. The left figure shows the values $\mu_S(x_i,x_9)$, the right shows the respective elements of $S$.}\label{fig:pariton}
\end{figure}\noindent

A positive integer $x$ is said to be \emph{a singular number} if there exists a GCD closed set $S=\{x_1,\ldots,x_n\}$, where $1\leq x_1<\cdots<x_n=x$, such that $\Psi_{S,\frac{1}{N}}(x)=0$. Otherwise $x$ is a \emph{nonsingular number}. Moreover, $x$ is a \emph{primitive singular number} if $x$ is singular and $x'$ is nonsingular number for all $x'\,|\,x$, $x'\neq x$.

Hong \cite{Ho3} conjectured that there are infinitely many even primitive singular numbers. He has also presented the following conjecture about odd primitive singular numbers.

\begin{conjecture}[\cite{Ho3}, Conjecture 4.3]\label{conj2}
There does not exist an odd primitive singular number.
\end{conjecture}

The counterexample found in the proof of Theorem \ref{vastaesim} also implies that neither Conjecture \ref{conj2} is true.

\begin{corollary}
There exists an odd primitive singular number.
\end{corollary}

\begin{proof}
By the proof of Theorem \ref{vastaesim} we know that $1020180525$ is an odd singular number. If it is not primitive singular number itself, then it has a nontrivial factor which is an odd primitive singular number.
\end{proof}

\section{Lattice-theoretic approach to singularity of power LCM matrices with real exponent}

So far we have only been studying the singularity of the usual LCM matrices. Next we consider singularity of power LCM matrices from lattice-theoretic viewpoint. The one thing that we can be sure of is that it is difficult to find singular power LCM matrices in which the exponent is an integer greater than $1$. Thus it is only natural to ask how this situation changes when the exponent is allowed to be any positive \emph{real} number. It turns out that in some cases already the semilattice structure of $(S,|)$ tells a lot about the singularity of power LCM matrices of $S$. We begin our study with two illustrative examples.

\begin{example}\label{ex:chain}
Let $L=\{z_1,z_2,\ldots,z_n\}$ be a chain with $z_1\prec z_2\prec\cdots\prec z_n$ (see Figure \ref{fig:esimerkit} (a)), let $\alpha$ be any positive real number and let $S$ be any set of positive integers such that $(S,|)\cong(L,\preceq)$. Then by \eqref{eq:psi} we get $$\Psi_{{S},\frac{1}{N^\alpha}}(x_1)=\frac{\mu_{S}(x_1,x_1)}{x_1^\alpha}=\frac{1}{x_1^\alpha}>0,$$ and for $1<i\leq n$ we have
$$\Psi_{{S},\frac{1}{N^\alpha}}(x_i)=\frac{1}{x_{i}^\alpha}-\frac{1}{x_{i-1}^\alpha}<0.$$
Thus the power LCM matrix $[S]_{N^\alpha}=[\mathrm{lcm}(x_i,x_j)^\alpha]$ is invertible for all $\alpha>0$.
\end{example}

\begin{figure}[ht!]
\centering
\subfigure[Ex. \ref{ex:chain}]
{{
\begin{tikzpicture}[scale=0.7]
\draw (0,1)--(0,2.5);
\draw (0,3.5)--(0,4);
\draw [fill] (0,1) circle [radius=0.1];
\draw [fill] (0,2) circle [radius=0.1];
\draw [fill] (0,4) circle [radius=0.1];
\node [right] at (0,1) {$z_1$};
\node [right] at (0,2) {$z_2$};
\node at (0,3) {$\vdots$};
\node [right] at (0,4) {$z_n$};
\node [right] at (1.5,2) {\ };
\node [right] at (-1.5,2) {\ };
\end{tikzpicture}}
}
\qquad
\qquad
\qquad
\qquad
\subfigure[Ex. \ref{ex:timantti}]
{{
\begin{tikzpicture}[scale=0.7]
\draw (1,0)--(0.25,1)--(0.92,1.92);
\draw (1,0)--(1.75,1)--(1.08,1.92);
\draw [fill] (0.25,1) circle [radius=0.1];
\draw [fill] (1.75,1) circle [radius=0.1];
\draw [fill] (1,0) circle [radius=0.1];
\draw [fill] (1,2) circle [radius=0.1];
\node [right] at (1,0) {$z_1$};
\node [right] at (0.25,1) {$z_2$};
\node [right] at (1.75,1) {$z_3$};
\node [right] at (1,2) {$z_4$};
\end{tikzpicture}}
}
\caption{The Hasse diagrams of the semilattices in Examples \ref{ex:chain} and \ref{ex:timantti}.}\label{fig:esimerkit}
\end{figure}
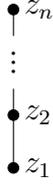
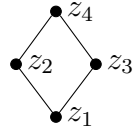

\begin{example}\label{ex:timantti}
Let $(L,\preceq)$ be the four element meet semilattice presented in Figure \ref{fig:esimerkit} (b). Suppose that $S=\{x_1,x_2,x_3,x_4\}=\{1,3,5,45\}$. Clearly $(S,|)\cong(L,\preceq)$. Let $\alpha$ be any positive real number. Applying \eqref{eq:psi} we obtain
\[
\Psi_{{S},\frac{1}{N^\alpha}}(1)=1,\quad \Psi_{{S},\frac{1}{N^\alpha}}(3)=\frac{1}{3^\alpha}-1\quad\text{and}\quad\Psi_{{S},\frac{1}{N^\alpha}}(5)=\frac{1}{5^\alpha}-1,
\]
which are all nonzero for all $\alpha>0$. However,
\[
\Psi_{{S},\frac{1}{N^\alpha}}(45)=\frac{1}{45^\alpha}-\frac{1}{5^\alpha}-\frac{1}{3^\alpha}+1,
\]
which is negative for $\alpha=\frac{1}{4}$ and positive for $\alpha=1$. Since $\Psi_{{S},\frac{1}{N^\alpha}}(45)$ is a continuous function of variable $\alpha$, this function must have zero value for some positive $\alpha_0$ (this $\alpha_0$ is located approximately at $0.328594$). It now follows from Proposition \ref{psi-lause} that the power LCM matrix $[S]_{N^{\alpha_0}}=[[x_i,x_j]^{\alpha_0}]$ is singular. This shows that our structure $(L,\preceq)$ does not possess the same property as chains were proven to have in our previous example.

Although we just found one set $S$ that yields a singular power LCM matrix for some positive real number $\alpha$, not every set of positive integers isomorphic to $(L,\preceq)$ has this property. To see this we only need to choose $S'=\{x_1',x_2',x_3',x_4'\}=\{1,3,5,15\}$. In this case we have
\[
\Psi_{{S'},\frac{1}{N^\alpha}}(i)=\Psi_{{S},\frac{1}{N^\alpha}}(i)\neq0\quad\text{for all}\ \alpha>0\ \text{and for all}\ i=1,2,3,
\]
but also
\[
\Psi_{{S'},\frac{1}{N^\alpha}}(15)=\frac{1}{15^\alpha}-\frac{1}{5^\alpha}-\frac{1}{3^\alpha}+1=\frac{1}{15^\alpha}(5^\alpha-1)(3^\alpha-1)\neq0
\]
for all $\alpha>0$. This means that the power LCM matrix $[S']_{N^{\alpha}}=[[x_i',x_j']^{\alpha}]$ is nonsingular for all $\alpha>0$.
\end{example}

As we saw in Example \ref{ex:chain}, sometimes the lattice-theoretic structure of $(S,|)$ alone tells us that the power LCM matrix of the set $S$ is invertible for all $\alpha>0$. On the other hand, Example \ref{ex:timantti} shows that in the remaining cases the information about the structure of $(S,|)$ is inconclusive and does not reveal whether or not all the power LCM matrices of the set $S$ are invertible. In this section our ultimate goal is to characterize all possible meet semilattices $(L,\preceq)$, whose structure is strong enough to guarantee the invertibility of the power LCM matrix for all GCD closed set $(S,|)\cong(L,\preceq)$ and for all $\alpha>0$. By making use of Lemma \ref{Möbius0} we are able to prove the following result, which brings us one step closer to achieving this aim.

\begin{theorem}\label{0-lause}
Let $(L,\preceq)$ be a meet semilattice with $n$ elements. Assume that there exist elements $x,y_1,\ldots,y_m\ (m\geq2)$ in $L$ such that $x$ covers $y_1,\ldots,y_m$ and $\mu_L(y,x)>0$, where $y=y_1\wedge\cdots\wedge y_k$. Then there exists a set $S=\{x_1,x_2,\ldots,x_n\}$ of positive integers and a positive real number $\alpha_0$ such that $(S,|)\cong (L,\preceq)$ and the power LCM matrix $[S]_{N^{\alpha_0}}=[[x_i,x_j]^{\alpha_0}]$ of the set $S$ is singular.
\end{theorem}

\begin{proof}
Let us denote $L=\{z_1,\ldots,z_n\}$, where $z_i\preceq z_j\Rightarrow i\leq j$ (in particular, $z_1=\min L$). We begin by constructing a GCD closed set $S'=\{x'_1,x'_2,\ldots,x'_n\}$ of positive integers such that $(S',|)\cong (L,\preceq)$. Let $p_2,p_3,\ldots,p_n$ be distinct prime numbers. We define $x_1'=1$ and 
\[
x_i'=p_i\mathrm{lcm}\{x_j'\,\big|\,j<i\text{ and }z_j\preceq z_i\}=\prod_{\substack{1\leq j\leq i\\z_j\preceq z_i}}p_j
\]
for $1<i\leq n$. It is easy to see that the set $S'$ is both GCD closed and isomorphic to $L$ (every element of $S'$ is either $1$ or a squarefree product of different primes).

Now suppose that $x_i'\in S'$ is an element such that it covers the elements $x_{i_1}',x_{i_2}',\ldots,x_{i_m}'\in S'$ and $\mu_{S'}(x_k',x_i')>0$, where $x_k'=x_{i_1}'\wedge x_{i_2}'\wedge\cdots\wedge x_{i_m}'$. Let $r$ be an arbitrary positive integer. Now let $S(r)=\{x_1,x_2,\ldots,x_n\}$, where
\[
x_j=
\left\{ \begin{array}{ll}
x_j' & \textrm{if}\ x_i'\nmid x_j',\\
p_i^rx_j' & \textrm{if}\ x_i'\,|\,x_j'.\\
\end{array}\right.
\]
Clearly $(S(r),|)\cong (S',|)\cong (L,\preceq)$.

Let $i$ be as fixed above. Then $x_i=p_i^rx_i'$. Let $r$ be sufficiently large (to be specified later). We define the function $h_{i,r}:\Rset\to\Rset$ by
\[
h_{i,r}(\alpha)=\Psi_{{S(r)},\frac{1}{N^\alpha}}(x_i)=\sum_{j=1}^i \frac{\mu_{S(r)}(x_j,x_i)}{x_j^\alpha}.
\]
By Lemma \ref{Möbius0} we know that $\mu_{S(r)}(x_j,x_i)=0$ for all $x_j\not\in\boldsymbol{\lsem}x_k,x_i\boldsymbol{\rsem}$. Thus the function $h_{i,r}$ comes to the form
\begin{align*}
h_{i,r}(\alpha)&=\sum_{x_k\,|\,x_j\,|\,x_i} \frac{\mu_{S(r)}(x_j,x_i)}{x_j^\alpha}=\frac{1}{x_k^\alpha}\sum_{a\,|\,\frac{x_i}{x_k}} \frac{\mu_{S(r)}(ax_k,x_i)}{a^\alpha}\\&=\frac{1}{x_k^\alpha}\left(\mu_{S(r)}(x_k,x_i)+\sum_{1\neq a\,|\,\frac{x_i}{x_k}} \frac{\mu_{S(r)}(ax_k,x_i)}{a^\alpha}\right).
\end{align*}
We are going to show that the factor on the right goes to zero for some $\alpha$. Here we have
\[
\lim_{\alpha\to\infty}(x_k^\alpha (h_{i,r}(\alpha))=\mu_{S(r)}(x_k,x_i)+\lim_{\alpha\to\infty}\sum_{1\neq a\,|\,\frac{x_i}{x_k}} \underbrace{\frac{\mu_{S(r)}(ax_k,x_i)}{a^\alpha}}_{\to 0\ \text{as}\ \alpha\to\infty}=\mu_{S(r)}(x_k,x_i)>0.
\]
The definition of the Möbius function $\mu_{S(r)}$ implies that
\[
x_k^0(h_{i,r}(0))=\sum_{x_k\,|\,x_j\,|\,x_i} \mu_{S(r)}(x_j,x_i)=\delta_{S(r)}(x_k,x_i)=0,
\]
since $x_k\neq x_i$. In addition,
\begin{align*}
&\frac{d(x_k^\alpha h_{i,r}(\alpha))}{d\alpha}=\sum_{1\neq a\,|\,\frac{x_i}{x_k}} -\log(a)\frac{\mu_{S(r)}(ax_k,x_i)}{a^\alpha}\\
&=\left(\sum_{\substack{a\,|\,\frac{x_i}{x_k}\\a\neq1,\frac{x_i}{x_k}}}-\log(a)\frac{\mu_{S(r)}(ax_k,x_i)}{a^\alpha}\right)-r\log(p_i)\log\left(\frac{x_i'}{x_k}\right)\frac{\mu_{S(r)}(x_i,x_i)}{\left(\frac{x_i}{x_k}\right)^\alpha}.
\end{align*}
Thus when the integer $r$ is sufficiently large, we have
\begin{align*}
\frac{d(x_k^\alpha h_{i,r}(\alpha))}{d\alpha}(0)=\sum_{\substack{a\,|\,\frac{x_i}{x_k}\\a\neq1,\frac{x_i}{x_k}}}&-\log(a)\mu_{S(r)}(ax_k,x_i)\\ &-r\log(p_i)\underbrace{\log\left(\frac{x_i'}{x_k}\right)}_{>0}\underbrace{\mu_{S(r)}(x_k,x_i)}_{>0}<0.
\end{align*}
Thus the function $x_k^\alpha h_{i,r}(\alpha)$ obtains negative values for some positive $\alpha$. In addition, $x_k^\alpha h_{i,r}(\alpha)$ is continuous. Now it follows from Bolzano's Theorem that there exists $\alpha_0\in]0,\infty[$ such that $x_k^{\alpha_0} h_{i,r}(\alpha_0)=0$ and therefore $h_{i,r}(\alpha_0)=\Psi_{{S(r)},\frac{1}{N^\alpha}}(x_i)=0$. Proposition \ref{psi-lause} now implies the matrix $[S(r)]_{N^{\alpha_0}}$ has to be singular.
\end{proof}

A subset $S$ of a meet semilattice is said to be \emph{a $\wedge$-tree set} if the Hasse diagram of the meet closure of $S$ is a tree (when considered as an undirected graph). An alternative way of putting this is that every element of the meet closure of $S$ covers at most one element of $\mathrm{meetcl}(S)$ (see \cite[Lemma 4.1]{MH14} for further characterizations). If the set $S$ is meet closed, then $S$ is a $\wedge$-tree set if and only if every element of $S$ covers at most one element of $S$.

Now we are finally in a position to prove the following theorem, which gives us the desired classification of finite meet semilattices.

\begin{theorem}\label{ehtolause}
Let $(L,\preceq)$ be a meet semilattice with $n$ elements, where $L=\{z_1,z_2,\ldots,z_n\}$. Then the following conditions are equivalent:
\begin{enumerate}
\item The LCM matrix $([x_i,x_j]^\alpha)$ is nonsingular for all $\alpha>0$ and for all sets $S=\{x_1,x_2,\ldots,x_n\}\subset\Zset^+$ such that $(S,|)\cong (L,\preceq)$. 
\item $L$ is $\wedge$-tree set.
\item For all $z_i,z_j\in L$
\[
\mu_L(z_i,z_j)>0\Rightarrow z_i=z_j.
\]
\end{enumerate}
\end{theorem}
\begin{proof}
$(1)\Rightarrow(2)$ First we assume Condition 1. Suppose for a contradiction that for some $i$, $z_i$ covers elements $z_{i_1},\ldots,z_{i_k}\in L$, where $k\geq 2$. Let $z_r=z_{i_1}\wedge\cdots\wedge z_{i_k}$. If $\mu_L(z_r,z_i)>0$, then Theorem \ref{0-lause} would imply that the matrix $([x_i,x_j]^\alpha)$ is singular for some $\alpha>0$ and $S\subset\Zset_+$, where $(S,|)\cong(L,\preceq)$. Thus we have
\[
\mu_L(z_r,z_i)=-\sum_{z_r\preceq z_j\prec z_i}\mu_L(z_r,z_j)\leq 0.
\]
Let $z_{l_1},\ldots,z_{l_m}\in L$ be the elements that cover $z_r$. Here $m\geq2$, since otherwise we would have $z_{l_1}\preceq z_{i_1},\ldots, z_{i_k}$ and further $z_r\prec z_{l_1}\preceq z_{i_1}\wedge\cdots\wedge z_{i_k}$. So we know that the terms $\mu_L(z_r,z_{l_1}),\ldots,\mu_L(z_r,z_{l_m})$ appear in the nonnegative sum
\begin{align*}
0\leq&\sum_{z_r\preceq z_j\prec z_i}\mu_L(z_r,z_j)\\
&=\mu_L(z_r,z_r)+\mu_L(z_r,z_{l_1})+\cdots+\mu_L(z_r,z_{l_m})+\sum_{z_{l_1},\ldots,z_{l_m}\prec z_j\prec z_i}\mu_L(z_r,z_j)\\
&=1-m+\sum_{z_{l_1},\ldots,z_{l_m}\prec z_j\prec z_i}\mu_L(z_r,z_j).
\end{align*}
Therefore there exists $z_j\succ z_r$ such that $\mu_L(z_r,z_j)>0$. Let $z_{k_1},\ldots,z_{k_t}$ be the elements that are covered by $z_j$ and are located on the interval $\boldsymbol{\lsem}z_r,z_i\boldsymbol{\rsem}$. Since $z_r$ is a lower bound for the elements $z_{k_1},\ldots,z_{k_t}$, we have $z_r\preceq z_{k_1}\wedge\cdots\wedge z_{k_t}$. Furhermore, $z_r$ has to be equal the meet of these elements covered by $z_j$, or otherwise we would have $\mu_L(z_r,z_j)=0$ by Lemma \ref{Möbius0}. Thus it follows from Theorem \ref{0-lause} that there exists set $S$ such that $(S,|)\cong(L,\preceq)$ and the power LCM matrix $([x_i,x_j]^\alpha)_{ij}$ of the set $S$ is singular for some $\alpha>0$. This is a contradiction with Condition 1.

$(2)\Rightarrow(3)$ Suppose then that Condition 2 holds. Let $z_i,z_j\in L$ with $\mu_L(z_i,z_j)>0$. Here we must have $z_i\preceq z_j$. Since $S$ is $\wedge$-tree set, the interval $\boldsymbol{\lsem}z_i,z_j\boldsymbol{\rsem}$ is a chain (see \cite[Lemma 4.1]{MH14}). In addition, the interval $\boldsymbol{\lsem}z_i,z_j\boldsymbol{\rsem}$ cannot have more than one element on it, since otherwise we would have $\mu_L(z_i,z_j)=-1$ (in the case when there are two elements on the interval) or $\mu_L(z_i,z_j)=0$ (in the case when there are more than two elements on the interval). Thus Condition 3 is satisfied.

$(3)\Rightarrow(1)$ For the last we assume Condition 3. Let $S$ be any subset of positive integers such that $(S,|)\cong(L,\preceq)$. If $x_i=\min S$, then 
\[
\Psi_{{S},\frac{1}{N^\alpha}}(x_i)=\sum_{j=1}^i \frac{\mu_{S}(x_j,x_i)}{x_j^\alpha}=\frac{1}{x_i^\alpha}>0.
\]
If $x_i\neq\min S$, then there is at least one element $x_k$ that is covered by $x_i$ and we obtain
\[
\Psi_{{S},\frac{1}{N^\alpha}}(x_i)=\sum_{j=1}^i \frac{\mu_{S}(x_j,x_i)}{x_j^\alpha}=\underbrace{\frac{1}{x_i^\alpha}-\frac{1}{x_k^\alpha}}_{<0}+\underbrace{\sum_{\substack{x_j\,|\,x_i\\x_j\neq x_i,x_k}} \frac{\mu_{S}(x_j,x_i)}{x_j^\alpha}}_{\leq0}<0.
\]
Thus the matrix $[S]_{N^\alpha}=([x_i,x_j]^\alpha)$ in invertible for all $\alpha>0$ and we have proven Condition 1.
\end{proof}

\section{Discussion about conjectures on singularity of power LCM matrices with real exponent}

What comes to nonsingularity of power GCD and LCM matrices, there are several open conjectures presented by Hong to be found in the literature. At this point we are ready to take a closer look at them. Let us begin with the following two.

\begin{conjecture}[\cite{Ho2}, Conjecture 4.1]\label{conj3}
Let $\alpha\neq0$ and let $S=\{x_1,\ldots,x_n\}$ be an odd-gcd-closed set. Then the matrix $[[x_i,x_j]^\alpha]$ on $S$ is nonsingular.
\end{conjecture}

\begin{conjecture}[\cite{Ho2}, Conjecture 4.5]\label{conj4}
Let $\alpha\neq0$ and let $S=\{x_1,\ldots,x_n\}$ be an odd-lcm-closed set. Then the matrix $[[x_i,x_j]^\alpha]$ on $S$ is nonsingular.
\end{conjecture}

First we should note that every counterexample for Conjecture \ref{conj3} generates a counterexample for Conjecture \ref{conj4} (in fact these two conjectures are equivalent to each other). In order to see this we utilize a method similar to that presented in \cite{HS}. Let $S=\{x_1,x_2,\ldots,x_n\}$ be a GCD closed set of odd positive integers such that $x_i\,|\,x_n$ for all $i=1,\ldots,n$. Now let $S'=\{\frac{x_n}{x_1},\frac{x_n}{x_2},\ldots,\frac{x_n}{x_n}\}$. The elements of $S'$ are clearly odd, and since $\gcd(x_i,x_j)\in S$ for all $i,j\in\{1,\ldots,n\}$, we have
\[ 
\textrm{lcm}\left(\frac{x_n}{x_i},\frac{x_n}{x_j}\right)=\frac{x_n^2}{x_ix_j\gcd\left(\frac{x_n}{x_i},\frac{x_n}{x_j}\right)}=\frac{x_n}{\gcd(x_j,x_i)}\in S'
\]
for all $i,j=1,\ldots,n$. Thus the set $S'$ is LCM closed. Furthermore, if $$\det [S]_{N^{\alpha}}=\det[[x_i,x_j]^{\alpha}]=\left(\prod_{k=1}^nx_k^{2\alpha}\right)\det\left[\frac{1}{(x_i,x_j)^\alpha}\right]=0,$$ we have $\det[\frac{1}{(x_i,x_j)^\alpha}]=0$ and therefore
\[
\det [S']_{N^{\alpha}}=\det\left[\left[\frac{x_n}{x_i},\frac{x_n}{x_j}\right]^\alpha\right]=x_n^n\det\left[\frac{1}{(x_i,x_j)^\alpha}\right]=0.
\]

It turns out that the elements of $S$ being odd has very little to do with the nonsingularity of the matrix $[[x_i,x_j]^\alpha]$. It follows already from Theorem \ref{vastaesim} that Conjecture \ref{conj3} does not hold for $\alpha=1$. More counterexamples can be found by using the method presented in the proof of Theorem \ref{0-lause} (the elements of $S$ can easily be chosen to be odd by assuming that $p_i\neq 2$ for all $i=2,\ldots,n$, as done in Example \ref{ex:timantti}). This means that for each semilattice structure $(L,\preceq)$, where $L$ is not a $\wedge$-tree set, there exist infinitely many counterexamples. Another consequence of Example \ref{ex:timantti} is that Theorem 1.5 in \cite{Ho2} cannot be improved; the condition ``$\epsilon<0$ or $\epsilon\geq1$'' cannot be improved to ``$\epsilon\neq0$''.

When applying Theorem \ref{0-lause} in practice the exponent $\alpha_0$ (for which the matrix $[[x_i,x_j]^\alpha]$ is singular) is often located near zero. This leaves open the possibility that Conjecture \ref{conj3} could be true when $\alpha>1$. Unfortunately not even this assumption is enough to salvage Conjecture \ref{conj3}. This can be seen by modifying the counterexample in Theorem \ref{vastaesim}. Let us consider the set
\[
S=\{1, 3, 5, 7, 195, 291, 1407, 4025q, 1020180525q\},
\]
where $q>1$ is an odd number. This set is clearly GCD closed, and thus we may define
\begin{align*}
&h_{9,q}(\alpha)=\Psi_{S,\frac{1}{N^\alpha}}(1020180525q)\\
&=\frac{1}{(1020180525q)^\alpha}-\frac{1}{(4025q)^\alpha}-\frac{1}{1407^\alpha}-\frac{1}{291^\alpha}-\frac{1}{195^\alpha}+\frac{1}{7^\alpha}+\frac{1}{5^\alpha}+\frac{2}{3^\alpha}-1\\
&=\frac{1}{q^\alpha}\underbrace{\left(\frac{1}{1020180525^\alpha}-\frac{1}{4025^\alpha}\right)}_{<0}-\frac{1}{1407^\alpha}-\frac{1}{291^\alpha}-\frac{1}{195^\alpha}+\frac{1}{7^\alpha}+\frac{1}{5^\alpha}+\frac{2}{3^\alpha}-1.
\end{align*}
Now let $\alpha=1$. By Example \ref{vastaesim} we know that if also $q=1$, then $h_{9,q}(1)=0$. But since $q>1$, $\frac{1}{q^\alpha}<1$ and we must have $h_{9,q}(1)>0$. Keeping in mind that $h_{9,q}(\alpha)$ is a continuous function on $\alpha$ and that in this case $\lim_{\alpha\to\infty}h_{9,q}(\alpha)=-1$, we now may conclude that there exists a real number $\alpha_0>1$ such that the matrix $[S]_{N^{\alpha_0}}=[[x_i,x_j]^{\alpha_0}]$ is singular.

In addition to those conjectures that already have been discussed, it turns out that the results found in this article have interesting consequences to some other conjectures found in the literature. Since the function $N^\alpha$ is clearly both completely multiplicative and strictly monotonous, it is easy to see that both of the following two conjectures are in fact false as well.

\begin{conjecture}[\cite{Ho2}, Conjecture 4.3]
Let $S=\{x_1,\ldots,x_n\}$ be an odd-gcd-closed set and $f$ a completely multiplicative function. If $f$ is strictly monotonous function, then the matrix $[f[x_i,x_j]]$ is nonsingular.
\end{conjecture}
\begin{conjecture}[\cite{Ho2}, Conjecture 4.7]
Let $S=\{x_1,\ldots,x_n\}$ be an odd-lcm-closed set and $f$ a completely multiplicative function. If $f$ is strictly monotonous function, then the matrix $[f[x_i,x_j]]$ is nonsingular.
\end{conjecture}

{\noindent\bf Acknowledgement} The authors wish to thank Kerkko Luosto for his valuable comments during the writing of this article.

\end{document}